\documentclass[11pt,reqno]{amsart}

\makeatletter
\usepackage{amssymb}
\usepackage{amssymb,amsmath,color,comment}
\usepackage{amsbsy}
\usepackage{amsfonts}

\def\marginpar#1{\ignorespaces}

\newtheorem{theorem}[equation]{Theorem}
\newtheorem{proposition}[equation]{Proposition}
\newtheorem{lemma}[equation]{Lemma}
\newtheorem{corollary}[equation]{Corollary}
\newtheorem{definition}[equation]{Definition}

\theoremstyle{definition}
\newtheorem{remark}[equation]{Remark}

\newtheorem{example}[equation]{Example}

\numberwithin{equation}{section}

\def\AArm{\fam0 \rm}%
\newdimen\AAdi%
\newbox\AAbo%
\def\AAk#1#2{\setbox\AAbo=\hbox{#2}\AAdi=\wd\AAbo\kern#1\AAdi{}}%

\newcommand{\BBone}{{\ensuremath{{\AArm 1\AAk{-.8}{I}I}}}}

\def\eqref#1{(\ref{#1})}
\def\eqlabel#1{\def\@currentlabel{#1}}

\def\formula#1{\def\@tempa{#1}\let\@tempb\theequation\def\theequation{%
\hbox{#1}}\def\@currentlabel{(\theequation)}$$}
\def\endformula{\leqno\hbox{(\@tempa)}$$\@ignoretrue\let\theequation\@tempb}

\def\given{\hskip5\p@\relax\vrule\@width.4\p@\hskip5\p@\relax}

\newcommand{\open}[1]{%
\par\normalfont\topsep6\p@\@plus6\p@\trivlist\item[\hskip\labelsep\itshape#1%
\@addpunct{.}]\ignorespaces}

\DeclareRobustCommand{\close}[1]{%
  \ifmmode 
  \else \leavevmode\unskip\penalty9999 \hbox{}\nobreak\hfill
  \fi
  \quad\hbox{$#1$}}

\newlength{\toskip}

\def\CC{\mathcal C}

\def\<{\langle}
\def\>{\rangle}

\def \R {{\mathbb R}}
\def \Q {{\mathbb Q}}

\def \P {{\mathbb P}}

\def \E {{\mathbb E}}
\def \N {{\mathbb N}}

\def \L {{\mathbb L}}

\def \Var {\textrm{Var}}
\def \Ent {\textrm{Ent}}
\def \Osc {\textrm{Osc}}

\makeatother

\begin{document}
\date{\today}

\title[Log-concave diffusions]{Semi Log-Concave Markov Diffusions}

 \author[P. Cattiaux]{\textbf{\quad {Patrick} Cattiaux $^{\spadesuit}$ \, \, }}
\address{{\bf {Patrick} CATTIAUX},\\ Institut de Math\'ematiques de Toulouse. Universit\'e de Toulouse. CNRS UMR 5219. \\ 118 route de Narbonne, F-31062 Toulouse cedex 09.}
\email{cattiaux@math.univ-toulouse.fr}

\author[A. Guillin]{\textbf{\quad {Arnaud} Guillin $^{\diamondsuit}$}}
\address{{\bf {Arnaud} GUILLIN},\\ Laboratoire de Math\'ematiques, CNRS UMR 6620, Universit\'e Blaise Pascal, avenue des Landais, F-63177 Aubi\`ere.} \email{guillin@math.univ-bpclermont.fr}

\maketitle
 \begin{center}

 \textsc{$^{\spadesuit}$  Universit\'e de Toulouse}
\smallskip

\textsc{$^{\diamondsuit}$Universit\'e Blaise Pascal \quad and \quad Institut Universitaire de
France}
\smallskip

 \end{center}

\begin{abstract}In this paper we intend to give a comprehensive approach of functional inequalities for diffusion processes under some ``curvature'' assumptions. Our notion of curvature coincides with the usual $\Gamma_2$ curvature of Bakry and Emery in the case of a (reversible) drifted Brownian motion, but differs for more general diffusion processes. Our approach using simple coupling arguments together with classical stochastic tools, allows us to obtain new results, to recover and to extend already known results, giving in many situations explicit (though non optimal) bounds. In particular, we show new results for gradient/semigroup commutation in the log concave case. Some new  convergence to equilibrium in the granular media equation is also exhibited.
\end{abstract}
\bigskip

\textit{ Key words :}  Functional inequalities, transport inequalities, diffusion processes, coupling, convergence to equilibrium.
\bigskip

\textit{ MSC 2000 : 26D10, 35K55, 39B62, 47D07, 60J60.}
\bigskip

{\footnotesize\tableofcontents}

\section{\bf Introduction and main results.}\label{Intro}

In this paper we shall investigate some properties of time marginals (at time $T$ finite or infinite) of Markov diffusion processes satisfying some logarithmic semi-convexity like property. The properties we are interested in are functional inequalities (Poincar\'e, log-Sobolev) or transportation inequalities. We shall also give some consequences for the long time behavior of such processes. 

Our main tools are on one hand coupling techniques and on the other hand stochastic calculus. We shall mainly use the so called ``synchronous'' coupling, i.e. using the same Brownian motion, but we also give some new results by using the ``mirror'' coupling (or coupling by reflection) introduced by Lindvall and Rogers in \cite{LR}. The main stochastic tool is (a very simple form of) Girsanov theory and $h$-processes.

The use of coupling techniques for obtaining analytic estimates is far to be new. It is impossible (and dangerous) to give here, even an account of the existing literature (see however \cite{Wbook} and references therein). The use of Girsanov theory for this goal is not new too. We shall recall later some references. The conjunction of both techniques is not usual.

We deliberately decided to present in details the simplest situation of a Brownian motion with a gradient drift, for which almost everything is well known, and then to extend our method to new situations. Some specialists would certainly find that these parts of the present paper are lengthy, but we think that the understanding of how the method works in this simple case is an useful guide for generalizations.  
\medskip

The meaning of logarithmic semi-convexity will generalize the ``usual'' one we recall now. \\

Let $U$ be a smooth ($C^\infty$) potential defined on $\R^n$ and satisfying for some $K \in \R$, $$\textbf{(H.C.K)} \quad \textrm{for all $(x,y)$, } \quad \langle \nabla U(x)-\nabla U(y),x-y\rangle \, \geq  \, K \, |x-y|^2  \, .$$ This property is called $K$-semi-convexity of $U$. It is clearly equivalent to the convexity of $U(x)-K|x|^2$. We denote $\Upsilon(dx) = e^{-U(x)} dx$ the Boltzmann measure associated to the potential $U$. If $e^{-U}$ is $dx$ integrable, we also introduce the normalized $\mu(dx) = \frac{1}{Z_U} \, e^{-U(x)} \, dx$ which is a probability measure. If $U$ is semi-convex,  $\mu$ is said to be semi log-concave.

Consider first the diffusion process, given by the solution of the Ito stochastic differential system
\begin{eqnarray}\label{eqito}
dX_t & = &  dB_t - \frac 12 \, \nabla U(X_t) \, dt \, ;\\
\mathcal L(X_0) & = & \mu_0 \, . \nonumber
\end{eqnarray}
$B_.$ being a standard brownian motion. It is known that \eqref{eqito} has an unique non explosive strong solution, in particular on can build a solution on any probability space equipped with some brownian motion. This is an easy consequence of Hasminski's explosion test using the Lyapunov function $x \mapsto |x|^2$.\\
Usual notations are in force: for a nice enough $f$, $P_t f(x) = \E(f(X_t^x))$ where $X_.^x$ denotes a solution such that $\mu_0=\delta_x$; $L$ denotes the infinitesimal generator i.e. $$L = \frac 12 \, \Delta \, - \frac 12  \, \langle \nabla U, \nabla \rangle\, ,$$ and $\Gamma$ denotes the carr\'e du champ, namely here $$\Gamma(f,g) = \frac 12 \, \langle \nabla f \, , \,  \, \nabla g \rangle \, \textrm{ and for simplicity } \Gamma(f)= \Gamma(f,f) \, .$$ $\P_{\mu_0}$ will denote the law of the solution of \eqref{eqito}, abridged in $\P_x$ when $\mu_0=\delta_x$, i.e. $\P$ is defined on the usual space $\Omega$ of continuous paths; $\mu_t$ will denote the law of $X_t$ for $t\geq 0$ and $P(t,x,.)$ denotes the law of $X_t^x$.

It is known that $\Upsilon$ is a symmetric (reversible) measure for the diffusion process, and is actually the unique invariant (stationary) measure for the process. If $\Upsilon$ is bounded, $\mu$ is ergodic.

In the latter case, $P_t$ is thus a symmetric semi-group on $L^2(\mu)$. The domain $D(L)$ of its generator contains the algebra $\mathcal A$ generated by the constant functions and $C_c^\infty$. In particular, if $f\in \mathcal A$, $\partial_t P_tf = L P_tf = P_t Lf$ in $L^2(\mu)$, so that since $\partial_t-L$ is hypo-elliptic $(t,x) \mapsto P_t f(x) \in C^{\infty}$.

$L$ is the basic example of generator satisfying the celebrated $C(K/2,+\infty)$ Bakry-Emery curvature condition (see \cite{Ane}). Indeed if we define $$\Gamma_2(f)=\frac 12 \, \left(L \Gamma(f) - 2 \Gamma(f,Lf)\right) \, ,$$ (H.C.K) is equivalent to $\Gamma_2(f)\geq (K/2) \, \Gamma(f)$.

This curvature condition is known to imply (and is in fact equivalent to)  a lot of nice inequalities for the semi-group, in particular  for all $T>0$ and all $x$, a commutation between $\Gamma$ and the semi group $P_t$ holds, namely
\begin{equation}\label{comGamma}
\Gamma(P_T f)\le e^{-KT} P_T\left(\sqrt{\Gamma(f)}\right)^2,
\end{equation}
which in turn implies powerful functional inequalities such as
\begin{equation}\label{eqproplogsob}
\textrm{$P(T,x,.)$ satisfies a log-Sobolev inequality with constant  \, $\frac 2K \, (1-e^{-KT})$.} 
\end{equation}

Recall that $\nu$ satisfies a log-Sobolev inequality with constant $C_{LS}$ if 
\begin{equation}\label{eqlogsob}
\Ent_\nu(f) := \, \int f^2 \, \log(f^2) \, d\nu  \, - \left(\int f^2 d\nu\right) \, \log \left(\int f^2 d\nu\right) \, \leq \, C_{LS} \, \int |\nabla f|^2 \, d\nu \, .
\end{equation} 
\eqref{eqproplogsob} is exactly what is (a little bit improperly) called a ``local'' log-Sobolev inequality in \cite{Ane} (theorem 5.4.7). For further informations and more, see the forthcoming book \cite{BaGL2}. The reader has to be careful with the constants, since we are writing the log-Sobolev inequality (as well as all other inequalities) in the usual form, while the ``Bakry-Emery'' version uses $\Gamma$, hence with an extra factor 2.

It is well known that a log-Sobolev inequality implies a Poincar\'e inequality
\begin{equation}\label{eqpoinc}
\Var_\nu(f):= \, \int f^2 \, d\nu  \, - \left(\int f d\nu\right)^2  \, \leq \, C_{P} \, \int |\nabla f|^2 \, d\nu \, ,
\end{equation} 
with $C_P = \frac 12 \, C_{LS}$, as well as a $T_2$ transportation inequality
\begin{equation}\label{eqt2}
W_2^2(\eta,\nu) \, \leq \, C_W \, H(\eta|\nu) \, ,
\end{equation}
with $C_W = \frac 12 \, C_{LS}$. Here $W_2$ denotes the Wasserstein distance between the probability measures $\eta$ and $\nu$, i.e. $$W_2^2(\eta,\nu) = \frac 12 \, \inf_\pi \, \int \, |x-y|^2 \, \pi(dx,dy) \, ,$$ where $\pi$ is a coupling of $\eta$ and $\nu$ (i.e. has respective marginals equal to $\eta$ and $\nu$) and $$H(\eta|\nu) \, = \, \int \left(\frac{d\eta}{d\nu}\right) \, \log \, \left(\frac{d\eta}{d\nu}\right) d\nu \, ,$$ denotes the Kullback-Leibler information or relative entropy of $\eta$ w.r.t. $\nu$. The latter property is due to Otto-Villani \cite{OV00}. Another approach and related properties were developed by Bobkov, Gentil and Ledoux \cite{BGL}. For a nice survey on transportation inequalities we refer to \cite{GL}. One can find in all these references another remarkable consequence of semi log-concavity, namely that a log-Sobolev inequality derives from a transportation inequality. This is a consequence of the following (H.W.I) inequality that holds for any nice $\mu$ density of probability $h$,
$$\textbf{(H.W.I)} \textrm{ If (H.C.K) holds then } \, \, H(h\mu|\mu) \leq \left(2 \, \int \, \frac{|\nabla h|^2}{h} \, d\mu\right)^{\frac 12} \, W_2(h\mu,\mu) - K \, W^2_2(h\mu,\mu) \, . $$ As a consequence, if (H.C.K) holds for some $K\leq 0$, a $T_2$ transportation inequality for $\mu$ implies a log-Sobolev inequality with constant $C_{LS} \leq (4/C_W) \, (1+(K/C_W))^{-2}$ provided $1+(K/C_W)>0$, in particular if $K=0$. \\
Let us finally remark that the starting point of this approach is the $\Gamma_2$ commutation property (\ref{comGamma}) which fails however to give a direct proof of the $T_2$ inequality.
\medskip

Our first goal is to show that functional and transportation inequalities can be derived, in the previous situation, by using synchronous coupling and simple tools of stochastic calculus. This is done in section \ref{secdrift}. The methods are then extended to a more general framework which is as natural for studying properties of time marginals as the $\Gamma_2$ framework.

Indeed, consider a classical diffusion process, given by the solution of an Ito stochastic differential system
\begin{eqnarray}\label{eqitogene}
dX_t & = & \sigma(X_t) \, dB_t + b(X_t) \, dt \, ;\\
\mathcal L(X_0) & = & \mu_0 \, . \nonumber
\end{eqnarray}
$B_.$ being a standard brownian motion. For simplicity we assume that $\sigma$ is a squared matrix. We extend (H.C.K) to this new situation 
$$\textbf{(H.C.K)} \quad \textrm{for all $(x,y)$, } \quad |\sigma(x)-\sigma(y)|_{HS}^2 + 2 \, \langle b(x)-b(y),x-y\rangle \, \leq  \, - K \, |x-y|^2  \, .$$ Notice that if $\sigma$ and $b$ are $C$-Lipschitz, i.e. each component is $C$ Lipschitz, (H.C.K) is satisfied for $K=-(C^2 n^2+C)$, but if $\sigma$ is $C$-Lipschitz, (H.C.K) can be satisfied for a non-negative $K$ provided $b$ is sufficiently repealing. Contrary to the case of a constant diffusion coefficient, (H.C.K) is not related to the Bakry-Emery curvature condition which involves in this situation controls on derivatives of higher order of the coefficients.

For simplicity in the sequel we shall assume that $\sigma \in C_b^2$ hence is $C$-Lipschitz and that $b$ is $C^2$, but not necessarily bounded nor with bounded derivatives. With these assumptions, once again if we assume that (H.C.K) is in force, \eqref{eqitogene} admits a unique non explosive strong solution using $x \mapsto x^2$ as a Lyapunov function for non explosion. We shall show this and other properties of the process in subsection \ref{subsecdiff}.

We still use the notations introduced before, but now $$L = \frac 12 \, \sigma \, \sigma^* \, \nabla^2 \, + \, b \nabla \, = \, \frac 12 \, a \, \nabla^2 \, + \, b \nabla \, ,$$ and $\Gamma$  the carr\'e du champ is now $$\Gamma(f,g) = \frac 12 \, \langle \sigma \, \nabla f \, , \, \sigma \, \nabla g \rangle \, .$$ 

Our first results can be gathered in Theorem \ref{thmcurve} below, after introducing some additional assumptions.

\textbf{Hypothesis (R).}
One of the following assumptions is satisfied (in addition to the fact that $\sigma \in C_b^2$ and $b\in C^2$).
\begin{enumerate}
\item[(R1)] \quad $\sigma=Id$ (or constant times $Id$),
\item[(R2)] \quad $\sigma$ and $b$ are $C_b^\infty$,
\item[(R3)] \quad $L$ is uniformly elliptic,
\item[(R4)] \quad $\sigma \in C_b^\infty$, $b \in C^\infty$ and has at most polynomial growth and $\partial_t - L$ is hypo-elliptic,
\item[(R5)] \quad $\sigma = a^{1/2}$, i.e. $\sigma$ is symmetric.
\end{enumerate}
Actually, assumptions (R1)-(R4) ensure that for $f \in \mathcal A$, $$(R6) \quad x \mapsto P_tf(x) \textrm{ is } C^2 \textrm{ and satisfies } \partial_t P_tf = LP_t f \, ,$$ which is what we really need. \\(R5) is a limiting situation for (R3) as we will see in (sub)subsection \ref{subsubsemi}. Of course the time marginal distributions only depend on $L$ and not on $\sigma$, but the constant $K$ is related to $a^{1/2}$ in (R5) (note that the Hilbert Schmidt norm of $\sigma(x)-\sigma(y)$ can change when we change $\sigma$ without modifying $a$).\\

\begin{theorem}\label{thmcurve}
Assume that (R) and (H.C.K) are satisfied. Let $M= \sup_{|u|=1} \,  \sup_x \, |\sigma(x) u|^2$. 
\begin{enumerate}
\item[(1)] The following commutation relation holds
\begin{equation}
|\nabla P_T f|^2\le e^{-KT}\,P_T|\nabla f|^2.
\end{equation}
\item[(2)] \quad If $\mu_0$ satisfies a Poincar\'e inequality with constant $C_P(0)$ then $\mu_T$ satisfies a Poincar\'e inequality with constant $$C_P(T)= e^{-KT} \, C_P(0) +\frac{M(1-e^{-KT})}{K} \, .$$  When $K=0$ one has to replace $\frac{(1-e^{-KT})}{K}$ by $T$. This applies in particular to $P(T,x,.)$ with $C_P(0)=0$.
\item[(3)] \quad $P(T,x,.)$ satisfies a $T_2$ transportation inequality with constant $C_T = C_T(M,K)$, bounded in time if $K>0$, linear in time if $K=0$ and exploding exponentially in time if $K<0$.
\item[(4)] \quad If $\mu_0$ satisfies $T_2$ with constant $C_W(0)$, $\mu_T$ satisfies $T_2$ with a constant $C_W(T)=C_W(T,M,K,C_W(0))$, bounded in time if $K>0$, linear in time if $K=0$ and exploding exponentially in time if $K<0$.
\item[(5)] \quad When $\sigma = Id$, if $\mu_0$ satisfies a log-Sobolev inequality with constant $C_{LS}(0)$, $\mu_T$ satisfies a log-Sobolev inequality with constant $$C_{LS}(T)=e^{-KT} \, C_{LS}(0) \, + \, \frac{2(1-e^{-KT})}{K} \, .$$ This applies in particular to $P(T,x,.)$ with $C_{LS}(0)=0$.
\end{enumerate}
\end{theorem}

Some other consequences, as for example convergence to equilibrium (when it exists) are also discussed in particular in subsection \ref{subsecconvgene}.

Of course, (1) is a weaker version of the commutation relation \eqref{comGamma} and (5) is nothing else than \eqref{eqproplogsob} when $2b=-\nabla U$. When a diffusion coefficient is present, (1) is however very different from the usual commutation property. For example we will show that it holds even in the negative infinite curvature case, but that it still enables us to provide interesting local functional inequalities. (3) as well as the general version of (H.C.K) we have introduced appeared (for this kind of application and to our knowledge) for the first time in the paper by Djellout, Guillin and Wu \cite{DGW}, Theorem 5.6 and condition 4.5 therein, for $K>0$. Our scheme of proof for the transportation inequality, based on Girsanov theory, is actually a simplified version of the one in \cite{DGW}, but instead of looking at the full law of the process on a time interval we shall use $h$-processes in order to look at time marginals. What we shall show is that the same scheme of proof also furnishes functional inequalities. This unified treatment of functional inequalities and transportation inequalities using an ad-hoc coupling is the novelty here. It easily extends to time dependent coefficients as shown in section \ref{secnonhom}.

In addition in this section we show how to directly obtain convergence to equilibrium and properties of the invariant measure for non linear diffusions of Mc Kean-Vlasov type, simplifying arguments in \cite{malrieu01}. 
\medskip

The use of stochastic calculus in deriving such inequalities is not new but only a small number of papers dealt with. One can trace back to the paper of Borell \cite{Bor}, who used Girsanov theory to study the propagation of log-concavity along the Schr\"{o}dinger dynamics (not the Fokker-Planck one we are looking at here). In addition to \cite{DGW} for transportation inequalities, one can also mention \cite{cat4,cat5} where similar ideas are used to study hyper-boundedness. More recently, using similar arguments, Lehec \cite{lehec} has studied gaussian functional inequalities and Fontbona and Jourdain \cite{FJ} obtained a pathwise version of the $\Gamma_2$ theory.
\medskip

Let us come back to \eqref{eqito}. (H.C.K) (or the $\Gamma_2$ theory) for $K>0$ applies to potentials $U$ which are the sum of $(K/2)|x|^2$ and of a concave (hence sub-linear) potential $V$. In particular, for ``super-convex'' potentials like $|x|^\beta$ with $\beta>2$, or more generally for (smooth) potentials $U$ which are uniformly convex at ``infinity'', (H.C.K) holds but with a negative $K$ due to the behavior of $U$ near the origin, so that, according to theorem \ref{thmcurve}, $\mu_T$ satisfies functional inequalities but with exploding constants in $T$. 

It is however well known, since $U=V+W$ with $V$ $K$-uniformly convex and $W$ bounded, that $\mu(dx) =e^{-U(x)} dx$ satisfies a log-Sobolev (and a Poincar\'e) inequality with a constant $C_{LS}=(2/K) \, \exp(\Osc W)$ where $\Osc$ denotes the oscillation of $W$. One can thus expect that $C_{LS}(T)$ is bounded in $T$.\\ In section \ref{secext}, we introduce the following extension of (H.C.K).

Let $\alpha$ be a non decreasing function defined on $\R^+$. We shall say that \textbf{(H.$\alpha$.K)} is satisfied for some $K>0$ if for all $(x,y)$ and all $\varepsilon >0$, $$\langle \nabla U(x) - \nabla U(y), x-y \rangle \geq K \, \alpha(\varepsilon) \, (|x-y|^2 - \varepsilon) \, .$$ When $\alpha(a)=1$, we may take $\varepsilon=0$ and we recognize (H.C.K). In Proposition \ref{propsuper} we show that $U(x)=|x|^{2\beta}$ (with $\beta \geq 1$) satisfies (H.$\alpha$.$K_\beta$) for $\alpha(a)=a^{\beta -1}$ and an explicit $K_\beta >0$. 

The main result of this section is then that, for suitable functions $\alpha$,  
\begin{center}
if (H.$\alpha$.K) holds (for $K>0$), then $\mu$ satisfies a log-Sobolev inequality. 
\end{center}
See theorems \ref{thmsuperconvex} and \ref{thmlack}. These theorems thus (partly) extend the Bakry-Emery criterion \eqref{eqproplogsob} to some non uniformly convex potentials. However, they are dealing with the invariant measure only and not with the law at time $T$ (only incomplete results are proved in this section for these distributions).
\medskip

The next section \ref{secreflect} is devoted to the use of mirror coupling. In a recent work \cite{Ebe}, Eberle has adapted the mirror coupling to get estimates of $W_1$ convergence for drifted brownian motions when the drift satisfies some ``convexity at infinity'' property. We recall Eberle's method and obtain some new consequences of his result. In addition, up to an extra condition, we show that his result (and all the consequences we derived) can be extended to general elliptic diffusion processes. We will also use this mirror coupling to show that we may get a weak version of the commutation property in the log concave case with the ``convexity at infinity'' property at least in dimension one, which is the first result we know of in this direction. Still in dimension one, we will also consider using mirror coupling for non linear diffusions.
\medskip

Section \ref{secpreserv} is peculiar. Using the results we have described for the Ornstein-Uhlenbeck process we show how to recover known results on the stability of functional inequalities under convolution (provided one of the terms is gaussian). 

\bigskip

\section{\bf Semi log-concave drifted brownian motion.}\label{secdrift}

In this first warming up section we shall look at the usual situation given by \eqref{eqito} 
\begin{eqnarray}\label{eqito2}
dX_t & = &  dB_t - \frac 12 \, \nabla U(X_t) \, dt \, ;\\
\mathcal L(X_0) & = & \mu_0 \, . \nonumber
\end{eqnarray}
and derive the classical inequalities.

\subsection{Commutation property}

 For functional inequalities the key is a commutation property of the gradient and the semi group. This commutation property is almost immediate using an appropriate coupling as explained below :
\begin{proposition}\label{propgrad}
In the situation of \eqref{eqito}, assume (H.C.K). Then for all $f \in \mathcal A$,
$$ W_2(P_t(x,\cdot),P_t(y,\cdot))\le e^{- Kt/2} |x-y|,$$
 \begin{equation}\label{commforte}
 |\nabla P_tf| \leq e^{- Kt/2} \, P_t|\nabla f| \, .\end{equation}
\end{proposition}
\begin{proof}
Applying Ito formula yields (almost surely)
\begin{eqnarray*}
e^{Kt} \, |X_t^x-X_t^y|^2 &=& |x-y|^2 + \int_0^t \, \left(K |X_s^x-X_s^y|^2 - \langle \nabla U(X_s^x)-\nabla U(X_s^y),X_s^x-X_s^y\rangle\right) \, e^{Ks} \, ds \\ &\leq& |x-y|^2 \, . 
\end{eqnarray*}
Hence, using the mean value theorem,  $$|P_tf(x)-P_tf(y)|\leq \E(|f(X_t^x)-f(X_t^y)|)\leq e^{-Kt/2} \, \E(|\nabla f(z_t)| \, |x-y|)$$ for some $z_t$ sandwiched by $X_t^x$ and $X_t^y$. It remains to use the continuity (and boundedness) of $\nabla f$ and the fact that $X_t^y$ goes almost surely to $X_t^x$ as $y \to x$ to conclude.
\end{proof}
\smallskip

\begin{remark}
As is seen from the proof, in fact, the sole convergence of the Wasserstein distance is not sufficient to get the commutation property exposed here. It will however be our starting point for the result when a diffusion coefficient is present. The synchronous coupling here enables us however to get an almost sure ``deterministic'' control of $X_t^x-X_t^y$ which is far more powerful. \hfill $\diamondsuit$
\end{remark}

\begin{remark}
We recall previously that (H.C.K) is exactly the $\Gamma_2$ condition of Bakry-Emery in this context, which is in fact equivalent to \eqref{commforte}. However the proof is very different from ours: it relies on a tricky calculus on $\psi(s)=e^{-Ks/2}\,P_s\sqrt{\Gamma(P_{t-s}f)}$ to show that $\psi'(s)\ge0$. \hfill $\diamondsuit$ 
\end{remark}

\begin{remark}\label{remwdecay}
If instead of $(x,y)$ the processes start with initial distribution $\pi_0$ the ``optimal coupling'' between $\mu_0$ and $\nu_0$ for the $W_2$ distance, the previous shows that $W_2^2(\mu_T,\nu_T) \leq e^{-KT} \, W_2^2(\mu_0,\nu_0)$. As discussed in the Appendix, this result can be used to show the existence and uniqueness of the invariant measure.\hfill $\diamondsuit$
\end{remark}
\medskip

\subsection{$h$-processes and functional inequalities.}\label{subsechp}

We now introduce the standard notion of $h$-process. Let $T>0$ and $h$ be a non-negative function such that $\int \, h \, d\mu_T = 1$. For simplicity, we assume for the moment that there exist $c$ and $C$ such that $C\geq h \geq c >0$.  We thus may define on the path-space up to time $T$ a new probability measure $$\frac{d\Q}{d\P_{\mu_0}}|_{\mathcal F_T} \, = \, h(\omega_T) \, .$$ It is immediately seen that $$\Q \circ \omega_s^{-1} = P_{T-s}h \, \mu_s \quad\textrm{ for all } \, 0\leq s \leq T  \, .$$ In this situation, it is well known (Girsanov transform theory) that one can find a progressively measurable process $u_s$ such that $$\frac{d\Q}{d\P_{\mu_0}}|_{\mathcal F_T} \, = \, P_Th(\omega_0) \, \exp\left(\int_0^T \, \langle u_s  ,  dM_s\rangle \, - \, \frac 12 \, \int_0^T \, |u_s|^2 \, ds\right) \, ,$$ where $\omega$ denotes the canonical element of the path-space and $M$ denotes the martingale part of $\omega$ under $\P_{\mu_0}$. In addition, it is easily seen (see e.g. \cite{CL94}) that 
\begin{equation}\label{eqentrop}
H(\Q|\P_{\mu_0}) \, = \, H(h \mu_T|\mu_T) \, = \, H(P_Th \mu_0|\mu_0) + \frac 12 \, \E^\Q\left(\int_0^T \, |u_s|^2 \, ds\right) \, .
\end{equation}
Actually, if $h \in \mathcal A$, it is immediate to check (applying Ito formula) that $$u_s = \nabla \, \log P_{T-s}h(\omega_s)$$ both $\P_{\mu_0}$ and $\Q$ almost surely.

\eqref{eqentrop} thus becomes
\begin{equation}\label{eqentrop1}
H(h \mu_T|\mu_T) \, = \, H(P_Th \mu_0|\mu_0) + \frac 12 \, \int_0^T \, \left(\int \, \frac{|\nabla P_s h|^2}{P_s h} \, d\mu_{T-s}\right) \, ds \, .
\end{equation}
If $h$ is smooth we may apply Proposition \ref{propgrad} in order to get
\begin{eqnarray}\label{eqlogsob1}
H(h \mu_T|\mu_T) &\leq& H(P_Th \mu_0|\mu_0) + \frac 12 \, \int_0^T \, \left(\int \, e^{-Ks} \, \frac{ P^2_s(|\nabla h|)}{P_s h} \, d\mu_{T-s}\right) \, ds  \, \nonumber \\ &\leq& H(P_Th \mu_0|\mu_0) + \frac 12 \, \int_0^T \, e^{-Ks} \, \left(\int  \, P_s \left(\frac{|\nabla h|^2}{h}\right) \, d\mu_{T-s}\right) \, ds \nonumber \\ &\leq& H(P_Th \mu_0|\mu_0) + \frac 12 \, \int_0^T \, e^{-Ks} \, \left(\int  \, \frac{|\nabla h|^2}{h} \, d\mu_{T}\right) \, ds \nonumber \\ &\leq& H(P_Th \mu_0|\mu_0) +  \, \frac{1-e^{-KT}}{2K} \, \int  \, \frac{|\nabla h|^2}{h} \, d\mu_{T} \, ,
\end{eqnarray}
where we have used Cauchy-Schwarz inequality for the second inequality and the Markov property for the third one. The previous inequality then extends to any $h$ in $C^1$ for which the right hand side makes sense, by density. We have thus obtained the following
\begin{proposition}\label{proplogsob}
In the situation of \eqref{eqito}, assume (H.C.K). If $\mu_0$ satisfies a log-Sobolev inequality with constant $C_{LS}(0)$, $\mu_T$ satisfies a log-Sobolev inequality with constant $$C_{LS}(T)=e^{-KT} \, C_{LS}(0) \, + \, \frac{2(1-e^{-KT})}{K} \, .$$ When $K=0$ one has to replace $\frac{(1-e^{-KT})}{K}$ by $T$. This applies in particular to $\mu_T=P(T,x,.)$ since $\delta_x$ satisfies a log-Sobolev inequality with constant equal to $0$.
\end{proposition}
\begin{proof}
Apply the log-Sobolev inequality to $\mu_0$. It furnishes (since $\int P_T h d\mu_0=1$), $$H(P_Th \mu_0|\mu_0) \leq \frac{C_{LS}(0)}{4} \, \int \, \frac{|\nabla P_Th|^2}{P_Th} \, d\mu_0 \leq e^{-KT} \, \frac{C_{LS}(0)}{4} \, \int  \, \frac{|\nabla h|^2}{h} \, d\mu_{T} \, ,$$ similarly as what we did in \eqref{eqlogsob1}. Hence the result applying \eqref{eqlogsob1}.
\end{proof}

 As we recalled in the introduction a log-Sobolev inequality implies a $T_2$ transportation inequality. It is interesting to see that one can directly obtain such an inequality for semi log-concave measures, by using the previous construction. But before to do this, just remark that the above proof using $h=1+\varepsilon g$ with $\int g d\mu_T=0$ allows us to obtain a similar result replacing the log-Sobolev inequality by a Poincar\'e inequality i.e.
\begin{proposition}\label{proppoinc}
In the situation of \eqref{eqito}, assume (H.C.K). If $\mu_0$ satisfies a Poincar\'e inequality with constant $C_{P}(0)$, $\mu_T$ satisfies a Poincar\'e inequality with constant $$C_{P}(T)=e^{-KT} \, C_{P}(0) \, + \, \frac{1-e^{-KT}}{K} \, .$$ When $K=0$ one has to replace $\frac{(1-e^{-KT})}{K}$ by $T$. This applies in particular to $\mu_T=P(T,x,.)$ since $\delta_x$ satisfies a Poincar\'e inequality with constant equal to $0$.
\end{proposition}

\begin{remark}
Once again, the proof presented here is very different from the one based on the $\Gamma_2$ calculus of Bakry-Emery which relies on the commutation property and the control of the derivative of $\psi(s)=P_s(P_{t-s}f\log(P_{t-s}f))$ to get a local logarithmic Sobolev inequality. Note that considering rather $\psi(s)=P_s((P_{t-s}f)^2)$ leads to a local Poincar\'e inequality. \hfill $\diamondsuit$
\end{remark}

\medskip

\subsection{Transportation inequalities.}\label{subsecthp}

The existence of $u_s$ and \eqref{eqentrop} are ensured as soon as $H(h\mu_T|\mu_T) < +\infty$ (see \cite{CL94}). For our goal we do not need the explicit expression of $u_s$. 

Indeed, Girsanov theory and Paul L\'evy characterization of brownian motion tell us that on $(\Omega,\Q)$, there exists some standard brownian motion $w$ (independent of $\omega_0$) such that, up to time $T$, $$\omega_t \, = \omega_0 + \,  w_t - \frac 12 \,  \int_0^t \, \nabla U(\omega_s) \, ds + \int_0^t \, u_s \, ds \, .$$ Since \eqref{eqito} has an unique \emph{strong} solution, one can build (on $(\Omega,\Q)$) a solution of $$z_t \, = z_0 + \, w_t - \frac 12 \,  \int_0^t \, \nabla U(z_s) \, ds  \, ,$$ the law of which being given by $$\P_{\nu_0} \quad \textrm{ with } \quad \nu_0 = \mathcal L(z_0) \, .$$ For instance we may choose $\nu_0=\mu_0$ or $z_0=\omega_0$ in which case $\nu_0=P_Th \mu_0$. But in all situations we choose the distribution of $(\omega_0,z_0)$ in such a way that $\E^{\Q}(|\omega_0-z_0|^2)=2 W_2^2(\nu_0,P_Th \mu_0)$ (or we take approximating sequences).

In particular $$z_t - \omega_t =  (z_0 -\omega_0) + \, \frac 12 \, \int_0^t \, (\nabla U(\omega_s)- \nabla U(z_s)) \, ds - \int_0^t \, u_s \, ds \, ,$$ $\Q$ almost surely. Applying Ito's formula and   (H.C.K) we obtain
\begin{eqnarray}\label{eqdist1}
\eta_t := \E^{\Q}(|z_t-\omega_t|^2) & \leq & \E^{\Q}(|z_0-\omega_0|^2) - \, K \, \int_0^t \, \eta_s \, ds + \, 2 \,  \int_0^t \, \E^{\Q}|\langle (z_s-\omega_s),u_s\rangle| \, ds \,  \nonumber\\ & \leq & \eta_0 \, - \, K \, \int_0^t \, \eta_s \, ds + 2  \, \left(\int_0^t \, \eta_s \, ds\right)^{\frac 12} \, \left(\E^{\Q}\left(\int_0^t \, |u_s|^2 \, ds\right)\right)^{\frac 12} \, \\ & \leq & \eta_0 \,  - \, K \, \int_0^t \, \eta_s \, ds + 2 \sqrt 2 \,  H^{\frac 12}(h\mu_T|\mu_T) \, \left(\int_0^t \, \eta_s \, ds\right)^{\frac 12} \, . \nonumber
\end{eqnarray}

We have then different alternatives depending on the sign of $K$. 

1) First in the case where $K>0$, one has using that $2ab\le Ka^2+b^2/K$
\begin{eqnarray*}
\eta_t&\le& \eta_0 \,  - \, K \, \int_0^t \, \eta_s \, ds + 2 \sqrt 2 \,  H^{\frac 12}(h\mu_T|\mu_T) \, \left(\int_0^t \, \eta_s \, ds\right)^{\frac 12} \\
&\le &\eta_0+\frac2K H^{\frac 12}(h\mu_T|\mu_T)
\end{eqnarray*}
so that we recover an uniform transportation inequality when $\eta_0=0$, which is moreover optimal for the invariant measure, considering logarithmic Sobolev inequality and Poincar\'e inequality. If $\mu_0$ satisfies some transportation inequality then one obtains that $\mu_T$ satisfies a transportation inequality with constant the sum of the initial constant plus $\frac2K$.\\
\medskip

2) The previous simple argument has however a serious drawback in the sense that in positive curvature, $\mu_T$ does not forget the ``initial measure''. Let us see how to deal with this problem. Start once again from the first estimation, but using It\^o's formula between $t$ and $t+\varepsilon$ and (H.C.K)
$$\eta_{t+\varepsilon}   \leq  \eta_t - \, K \, \int_t^{t+\varepsilon} \, \eta_s \, ds + \, 2 \, \int_t^{t+\varepsilon} \, \E^{\Q}|\langle (z_s-\omega_s),u_s\rangle| \, ds$$
so that we may differentiate in time to get for all positive $\lambda$
\begin{eqnarray*}
\eta'_t&\le&-K\eta_t+2\E^\Q|\langle (z_t-\omega_t),u_t\rangle| \, ds,\\
&\le&-(K+\lambda)\eta_t+\frac1\lambda\E^\Q|u_t|^2.
\end{eqnarray*}
Using Gronwall's lemma, we get that
$$
\eta_T \le e^{(-K+\lambda)T}\,\eta_0 +\frac1\lambda \int_0^Te^{(K-\lambda)(s-T)}\E^\Q|u_t|^2\,dt.
$$
so that if $K>0$ we get, for $\lambda<K$
$$\eta_T\le e^{(-K+\lambda)T}\,\eta_0 +\frac1\lambda \int_0^T\E^\Q|u_t|^2\,dt\le e^{(-K+\lambda)T}\,\eta_0 +\frac2\lambda\, H(h\mu_T|\mu_T).$$
Note that this is once again optimal for the limiting measure, and captures the fact that it forgets the initial condition. When $K<0$, we then have
$$\eta_T\le e^{(-K+\lambda)T}\,\eta_0+\frac2\lambda e^{(-K+\lambda)T}\, H(h\mu_T|\mu_T).$$
Note however the presence of the additional parameter $\lambda$. \\
\medskip

3) Let us see how a direct approach may get rid of this additional parameter, which is particularly important in negative curvature. Define $$a_t = e^{Kt} \, \int_0^t \, \eta_s \, ds \, - \, \frac{e^{Kt}}{K} \, \eta_0 \, .$$ We have $$a'_t \, \leq \, 2 \sqrt 2 \, e^{Kt/2} \, H^{\frac 12}(h\mu_T|\mu_T) \, \left(a_t + \frac{e^{Kt}}{K} \, \eta_0\right)^{\frac 12} \, .$$ Since $\frac{e^{Kt}}{K}\leq \frac{e^{KT}}{K}$  we obtain  
$$\left(a_t + \frac{e^{KT}}{K} \, \eta_0\right)^{\frac 12} \leq \left(a_0 + \frac{e^{KT}}{K} \, \eta_0\right)^{\frac 12} + 2\sqrt 2 \, \frac{e^{Kt/2}-1}{K} \, H^{\frac 12}(h\mu_T|\mu_T) \, .$$ It follows $$\left(\int_0^T \, \eta_s \, ds\right)^{\frac 12} \leq \left(\frac{1-e^{-KT}}{K} \, \eta_0\right)^{\frac 12} + 2\sqrt 2 \, \, \left(\frac{1- e^{-KT/2}}{K}\right) \, H^{\frac 12}(h\mu_T|\mu_T) \, .$$ 
For $K\geq 0$, this yields, since $W_2^2\left(h \mu_T,\nu_T\right) \leq \frac 12 \, \E^{\Q}(|z_t-\omega_t|^2)$, and using $\sqrt a \, \sqrt b \, \leq \frac 12 \, (a+b)$,
\begin{eqnarray}\label{eqtal1}
W_2^2\left(h \mu_T,\nu_T\right) &\leq& \, \left(1 + \sqrt 2 \, \frac{1-e^{-KT}}{K}\right) \, W_2^2\left(P_Th \mu_0,\nu_0\right) \nonumber \\ & & + \left(\frac{\sqrt 2}{2} + \frac{4 (1- e^{-KT/2})}{K}\right) \, H(h\mu_T|\mu_T)  \, .
\end{eqnarray}
If $\eta_0=0$, \eqref{eqtal1} can be improved in
\begin{equation}\label{eqtal2}
W_2^2\left(h \mu_T,\nu_T\right) \, \leq \, \frac{4 (1- e^{-KT/2})}{K} \, H(h\mu_T|\mu_T)  \, .
\end{equation}
When $K\leq 0$, we obtain
\begin{eqnarray}\label{eqtal3}
W_2^2\left(h \mu_T,\nu_T\right) &\leq&   \left(1 + \sqrt 2 \, \frac{1-e^{-KT}}{K} +2(e^{-KT}-1)\right) \, W_2^2\left(P_Th \mu_0,\nu_0\right) \nonumber \\ & & +  \, \left(\frac{\sqrt 2}{2} + 4 \, \frac{(1- e^{-KT/2})}{K} -  \, 4 \, \frac{(1-e^{-KT/2})^2}{K}\right) \, H(h\mu_T|\mu_T)  \, . 
\end{eqnarray}
Again if $\eta_0=0$, \eqref{eqtal3} can be improved in
\begin{equation}\label{eqtal4}
W_2^2\left(h \mu_T,\nu_T\right) \, \leq \, 4 \, \left(\frac{(1- e^{-KT/2})}{K} -  \, \frac{(1-e^{-KT/2})^2}{K}\right) \, H(h\mu_T|\mu_T)  \, .
\end{equation}

The previous inequalities then extend to any non-negative $h$ (not necessarily bounded below nor above).
\medskip

If we choose $\mu_0=\delta_x$, we have $\mu_T=P(T,x,.)$, $1=\int h d\mu_T =P_Th(x)$ and so $\nu_0=\delta_x$ and $\nu_T=\mu_T$. Hence
\begin{proposition}\label{propdelta}
In the situation of \eqref{eqito}, assume (H.C.K). Then $P(T,x,.)$ satisfies a $T_2$ transportation inequality $$W_2^2\left(h P(T,x,.),P(T,x,.)\right) \, \leq \, C_T \, H(h P(T,x,.)|P(T,x,.)) \, ,$$ with $$C_T =\min\left(\frac2K,\, \frac{4 (1- e^{-KT/2})}{K}\right)$$ when $K>0$, $2T$ when $K=0$ and $$C_T = \frac{4 (1- e^{-KT/2})}{K} - 4 \, \frac{(1-e^{-KT/2})^2}{K}$$  when $K\leq 0$.
\end{proposition}

If we choose $\nu_0=\mu_0$, we may use the convexity of $t \mapsto t \log t$, i.e $$H(P_Th \mu_0|\mu_0) = \int \, P_T h \, \log P_T h \, d\mu_0 \leq \int \, P_T (h \, \log h) \, d\mu_0 = H(h \mu_T|\mu_T) \, ,$$ in order to get
\begin{proposition}\label{proptrans}
In the situation of \eqref{eqito}, assume (H.C.K). If $\mu_0$ satisfies $T_2$ with constant $C(0)$, then $\mu_T$ satisfies $T_2$ with a constant $C(T)$ given, 
\begin{enumerate}
\item when $K>0$, for $0<\lambda<K$,
$$C(T)=e^{-(K-\lambda)T}C(0)+\frac2\lambda \, ,$$
\item and when $K\le0$, 
$C(T)=C_T + \frac{\sqrt 2}{2} + B_T \, C(0)$  with $$B_T = 1 + \sqrt 2 \, \frac{1-e^{-KT}}{K} +2(e^{-KT}-1) \, .$$ 
\end{enumerate}
\end{proposition}
\medskip

\begin{remark}\label{remunbound}
All what precedes holds even if $\Upsilon$ is not bounded (i.e. if the process is not positive recurrent), in which case of course, $K<0$. \hfill $\diamondsuit$
\end{remark}

\begin{remark}
If we choose $\mu_0=\mu$ (assuming that $\Upsilon$ is bounded), we have to choose $\nu_0=P_Th \mu$ hence $\nu_T=P_{2T}h \, \mu_0$. After noticing that we can slightly refine the previous bound replacing $H(h\mu_T|\mu_T)$ by $H(h\mu_T|\mu_T)-H(P_Th \mu_0|\mu_0)$ according to \eqref{eqentrop}, we obtain
$$ W_2(P_{2T}h\mu,h\mu) \, \leq \, \sqrt{C_T \, (H(h\mu|\mu)-H(P_Th \mu|\mu))}$$
and finally
\begin{equation}\label{eqtransinv1}
W_2(h\mu,\mu) \, \leq \, \sqrt{C_T \, (H(h\mu|\mu)-H(P_Th \mu|\mu))} + W_2(P_{2T}h\mu,\mu) \, .
\end{equation}
The latter has to be compared with remark 4.9 in \cite{CGJMP} which shows that the inequality $$W_2(h\mu,\mu) \, \leq \, \sqrt{T \, (H(h\mu|\mu)-H(P_Th \mu|\mu))} + W_2(P_{T}h\mu,\mu) \, $$ always holds. \hfill $\diamondsuit$
\end{remark}

\begin{remark}\label{remBE}
If $K>0$ we may let $T$ go to $+\infty$ in Proposition \ref{proplogsob} and recover that $\mu$ satisfies a log-Sobolev inequality with constant $2/K$, hence a $T_2$ transportation inequality with constant $1/K$ (in particular we are loosing a factor $4$ in Proposition \ref{propdelta}).\\ 
Similarly, when $T \to +\infty$, \eqref{eqtransinv1} shows that if $K>0$, $\mu$ satisfies a $T_2$ inequality , and since $\mu$ is log-concave, satisfies a log-Sobolev inequality. This scheme of proof does not require Proposition \ref{propgrad}, but the (H.W.I) inequality. Unfortunately it does not furnish the optimal constant. \hfill $\diamondsuit$
\end{remark} 
\medskip

\subsection{ Transportation-Fisher Inequalities.}

 Let us see look now at another type of Transportation Information inequality recently introduced in \cite{GLWY}, which is weaker but quite close to logarithmic Sobolev inequality (in fact equivalent under bounded curvature). 
We are obliged to come back to the initial inequality in \eqref{eqdist1} which becomes in our new situation
\begin{equation}\label{eqdist3comp}
\eta_t  \leq   \eta_0 \,  - \, K \, \int_0^t \, \eta_s \, ds + 2 \, \int_0^t \, \E^{\Q}\left(|z_s-\omega_s| \, |\nabla \log P_{T-s}h(\omega_s)|\right) ds  \, .
\end{equation}
Replacing the pair $(0,t)$ by $(t,t+\varepsilon)$ we thus have
\begin{eqnarray*}
\eta_{t+\varepsilon}  &\leq&   \eta_t \,  - \, K \, \int_t^{t+\varepsilon} \, \eta_s \, ds + 2 \, \int_t^{t+\varepsilon} \,  \eta^{\frac12}_s \, \left(\E^{\Q}\left(|\nabla \log P_{T-s}h(\omega_s)|^2\right)\right)^{\frac 12} \, ds \, \\ &\leq& \eta_t \,  - \, K \, \int_t^{t+\varepsilon} \, \eta_s \, ds + 2 \, \int_t^{t+\varepsilon} \,  \eta^{\frac12}_s \, \left(\int \frac{|\nabla P_{T-s}h|^2}{P_{T-s}h} \, d\mu_{s}\right)^{\frac 12} \, ds \, . 
\end{eqnarray*}
It follows that $t \mapsto \eta_t$ is differentiable and satisfies,  
\begin{eqnarray}\label{eqdist4comp}
\eta'_t &\leq& - \, K \, \eta_t + 2 \, \eta^{\frac12}_t \, \left(\int \frac{|\nabla P_{T-t}h|^2}{P_{T-t}h} \, d\mu_{t}\right)^{\frac 12} \nonumber \\ &\leq& - \, K \, \eta_t + 2 \, \eta^{\frac12}_t \, \left(\int \, \frac{P_{T-t}^2|\nabla h|}{P_{T-t}h} \, d\mu_{t}\right)^{\frac 12} \nonumber \\ &\leq& - \, K \, \eta_t +  2 \, \eta^{\frac12}_t \, \left(\int P_{T-t}\left(\frac{|\nabla h|^2}{h}\right) \, d\mu_{t}\right)^{\frac 12} \nonumber \\&\leq& - \, K \, \eta_t +  2 \, \eta^{\frac12}_t \, \left(\int \frac{|\nabla h|^2}{h} \, d\mu_{T}\right)^{\frac 12} \, .  
\end{eqnarray}
(for the second inequality, recall that (H.C.$0$) is satisfied so that, for short, $|\nabla P_s| \leq P_s |\nabla|$.) To explore \eqref{eqdist4comp} we shall use the usual trick $ab \leq \lambda a^2 + \frac 1\lambda \, b^2$ for $a,b,\lambda$ positive. Hence
\begin{equation}\label{eqdist5comp}
\eta'_t \leq \, \left( -K  + 2\lambda \right)\, \eta_t + \frac{2}{\lambda} \, \left(\int \frac{|\nabla h|^2}{h} \, d\mu_{T}\right)  \, . 
\end{equation}
We deduce, denoting $A=K \, - 2\lambda$,  
$$
W_2^2(h\mu_T,\mu_T) \leq  \eta_T \leq \eta_0 \, e^{-AT} + \frac{2(1-e^{-AT})}{A \lambda} \, \int \frac{|\nabla h|^2}{h} \, d\mu_{T}\, .
$$

This inequality is close to what is called a $W_2I$ inequality (see \cite{GL} definition 10.4 or \cite{GLWY} for examples and details on properties of WI inequality). Here we obtain a defective $W_2I$ inequality. However, as $T\to +\infty$, we recover the true $W_2I$ inequality for the invariant distribution, which together with the (H.W.I) inequality allows us to recover the log-Sobolev inequality. Nevertheless, we get

\begin{proposition}
\label{WI}
Assume (H.C.K), then $P(T,x,\cdot)$ satisfies a WI inequality of constant  $\frac{2(1-e^{-AT})}{A \lambda}$. If we suppose moreover that $\mu_0$ satisfies a WI inequality with constant $D(0)$ then $\mu_T$ satisfies a WI inequality with constant $D(T)=e^{-AT}D(0)+\frac{2(1-e^{-AT})}{A \lambda}$.
\end{proposition}
\medskip 
 
 \begin{remark}
 As remarked, under (H.C.K), the inequalities verified by the law $\mu_T$ depend on the inequalities verified by the initial measure, in the range between Poincar\'e and logarithmic Sobolev inequality. Indeed, a logarithmic Sobolev inequality implies a WI inequality, but to get the the WI inequality for $P_T$ we need only a WI inequality for the initial measure. As seen by the example of the Gaussian measure, which satisfies (H.C.K), no stronger inequalities can be obtained. \hfill $\diamondsuit$
 \end{remark}
 
\begin{remark}\label{rembridge}
Instead of the $h$-process one should consider Schr\"{o}dinger bridges allowing to choose both the initial and the final time marginals. Indeed if $h$ is bounded it is known that one can find non-negative functions such that the measure $$\frac{d\Q}{d\P_{\mu}}|_{\mathcal F_T} \, = \, f(\omega_0) \, g(\omega_T) \, ,$$ satisfies $$\Q \circ \omega_0^{-1} =  \mu \quad\textrm{ and } \, \Q \circ \omega_T^{-1} =  h \mu \, .$$ The pair $(f,g)$ satisfies $$f \, P_Tg = 1 \textrm{ and } g \, P_Tf =h \quad \mu \, \, \, a.s. \, ,$$ and the drift $u_s$ is given by $u_s = \nabla \, \log P_{T-s}g$. As before it is immediately checked that  
$$\Q \circ \omega_s^{-1} = P_s f \, P_{T-s}g \, \mu \quad\textrm{ for all } \, 0\leq s \leq T  \, .$$ For all this we refer to \cite{Fol} p.162 and \cite{CGamb} section 6. Even if $h$ is bounded from below, we do not know whether $g$ inherits this property. Nevertheless, at least formally we have the relation $$g \, P_T\left(\frac{1}{P_Tg}\right) = h \, .$$ Proceeding as before we obtain $$H(h\mu|\mu) \leq \frac{1-e^{-KT}}{2K} \, \int  \, \frac{|\nabla g|^2}{g} \, \frac hg \,  d\mu \,= \, \frac{1-e^{-KT}}{2K} \, \int  \, \frac{|\nabla g|^2}{g} \, P_T\left(\frac{1}{P_Tg}\right) \,  d\mu \, , $$ and 
\begin{equation}\label{eqtcouple}
W_2^2(\nu_{0T},\mu_{0T}) \, = \, W_2^2(h\mu,\mu) \, \leq \, C'_T \, H(\nu_{0T}|\mu_{0T}) \, , 
\end{equation}
where $\nu_{0T}$ (resp. $\mu_{0T}$) denotes the $\Q$ (resp. $\P_\mu$) joint law of $(X_0,X_T)$, i.e. $$\nu_{0T}(dx,dy)=f(x)g(y) \, \mu_{0T}(dx,dy) \, .$$ Unfortunately, these inequalities do not seem to give new results. \hfill $\diamondsuit$
\end{remark} 
\medskip

\begin{remark}\label{remdrift}

In almost all what we did we may replace the drift $- \frac 12 \, \nabla U$ by a general (non-gradient) smooth drift $b$ satisfying $$(H.C.K) \quad 2 \, \langle b(x)-b(y),x-y\rangle \leq - K \, |x-y|^2 \, .$$ 
All the results of this section remain true in this more general situation, as far as we do not use reversibility. The only result where we used reversibility actually is \eqref{eqtransinv1}. Indeed, if the initial law is $\nu_0 =P_Th \mu$, $\nu_T= P_T^*P_T h \, \mu$ where $P_T^*$ denotes the $\mu$ adjoint semi-group. 

The only delicate point is the smoothness of $P_th$ and the fact that $\partial_t (P_th)=LP_th$ in the usual sense. This will be discussed in an even more general setting in the next section, where we shall look at more general cases with non constant diffusion coefficient. 

Notice that, according to the discussion in the Appendix, if $K>0$, there exists an unique invariant probability measure $\mu_\infty$. Furthermore, for all $0<T\leq +\infty$, $\mu_T$ admits a density w.r.t. Lebesgue measure (whatever the initial distribution), and the convergence of the densities hold weakly in $\L^1$.
\hfill $\diamondsuit$
\end{remark}
\medskip

\section{\bf General diffusion processes.}\label{secdiff}

We shall now extend the results of the previous section to the general situation of \eqref{eqitogene},
\begin{eqnarray}\label{eqitogene2}
dX_t & = & \sigma(X_t) \, dB_t + b(X_t) \, dt \, .\\
\mathcal L(X_0) & = & \mu_0 \, . \nonumber
\end{eqnarray}
First of all we have to discuss some properties of the process and the associated quantities. As we said in the introduction, we need some regularity for $P_tf$ at least if $f\in \mathcal A$. So there is a technical price to pay. We decided to pay this price at the level of the study of the process, rather than in deriving inequalities.
\medskip

\subsection{Some properties of the process.}\label{subsecdiff}

\subsubsection{Non explosion.}\label{subsubexplo}
Since we assume that $\sigma \in C_b^2$, when (H.C.K) is fulfilled, $b$ satisfies $$2 \, \langle b(x)-b(y),x-y\rangle \, \leq  \, - D \, |x-y|^2 \, ,$$ for some $D \in \R$. In particular, $$2 \, \langle b(x),x\rangle \, \leq  \, - D \, |x|^2  + 2 \, |b(0)||x| \, .$$ Thus, if $S_k$ denotes the exit time from the ball $B(x,k)$, and $t_k=t\wedge S_k$ it holds
\begin{eqnarray*}
\E(|X^x_{t_k}|^2) &=& |x|^2 + \E\left(\int_0^{t_k} \, Trace(a(X_s^x)) + 2 \, \langle b(X_s^x),X_s^x\rangle \, ds\right) \\ &\leq& |x|^2 + Nt + |D| \, \int_0^{t} \E(|X^x_{s_k}|^2) \, ds + 2 \, |b(0)| \, \int_0^{t} \E(|X^x_{s_k}|) \, ds\\ &\leq& |x|^2 + (N+2 |b(0)|)t + (|D|+2 |b(0)|) \, \int_0^{t} \E(|X^x_{s_k}|^2) \, ds 
\end{eqnarray*}
where, since $\sigma$ is bounded, we have defined $$N= \parallel Trace(a(.))\parallel_{\infty} \, ,$$ and where we used $|y|\leq 1 + |y|^2$. Applying Gronwall lemma we obtain that $\E(|X^x_{t_k}|^2)$ is bounded independently on $k$, so that we may pass to the limit in $k$. This proves non explosion up to time $t$ (since the explosion time is the increasing limit of the sequence $S_k$) for all $t$.

It is then easily seen that one can perform similar calculations with $g(t,x)=\exp (e^{-Ct} |x|^2)$ for a large enough $C$ in order to kill the integrated term, i.e
\begin{lemma}\label{lemconcentration}
There exists a large enough $C_{e}>0$, such that $\E(\exp (e^{-C_{e}t} |X_t^x|^2)) \leq e^{|x|^2}$.
\end{lemma}

It is interesting to notice that one can similarly obtain some ``deviation'' bound from the starting point. Indeed
\begin{eqnarray*}
\E(|X^x_{t}-x|^2) &=&  \E\left(\int_0^{t} \, Trace(a(X_s^x)) + 2 \, \langle b(X_s^x),X_s^x-x\rangle \, ds\right) \\ &=& \E\left(\int_0^{t} \, Trace(a(X_s^x)) + 2 \, \langle b(X_s^x)-b(x)+b(x),X_s^x-x\rangle \, ds\right)\\
&\leq&  Nt - D \, \int_0^{t} \E(|X^x_{s}-x|^2) \, ds + 2 \, |b(x)| \, \int_0^{t} \E(|X^x_{s}-x|) \, ds\end{eqnarray*}
so that arguing as we did in order to get \eqref{eqtal1} and \eqref{eqtal3} we obtain the existence of constants $\alpha(T,D)$ and $\beta(T,D)$ such that for $0\leq t \leq T$,
\begin{equation}\label{eqdev}
\E(|X^x_{t}-x|^2) \leq (\alpha(T,D) N +\beta(T,D) |b(x)|^2) \, t \, .
\end{equation}
\smallskip

\subsubsection{Properties of the semi-group.}\label{subsubsemi}
Let us mimic what we did to get Proposition \ref{propgrad} i.e. apply Ito formula to get
\begin{eqnarray*}
 && e^{Kt} \, |X_t^x-X_t^y|^2 = |x-y|^2 + \int_0^t \, 2 e^{Ks} \langle \sigma(X_s^x)-\sigma(X_s^y),X_s^x-X_s^y\rangle \, dB_s \\ &+& \int_0^t \, \left(K |X_s^x-X_s^y|^2 + |\sigma(X_s^x)-\sigma(X_s^y)|_{HS}^2 + 2 \, \langle b(X_s^x)-b(X_s^y),X_s^x-X_s^y\rangle \right) \, e^{Ks} \, ds 
\end{eqnarray*}
so that, if (H.C.K) holds
\begin{equation}\label{eqgenedom}
e^{Kt} \, |X_t^x-X_t^y|^2 \leq |x-y|^2 + \int_0^t \, 2 e^{Ks} \langle \sigma(X_s^x)-\sigma(X_s^y),X_s^x-X_s^y\rangle \, dB_s  \, .
\end{equation}
Notice that with our assumptions, the right hand side of \eqref{eqgenedom} is a (true) martingale, so that 
\begin{equation}\label{eqborne}
\E(|X_t^x-X_t^y|^2) \leq e^{-Kt} \, |x-y|^2 \, .
\end{equation}

In summary, we get

\begin{theorem}\label{contW2gen}
Assume (R) and (H.C.K.) then we get that
\begin{equation}\label{W2gen}
W_2(P_t(x,\cdot),P_t(y,\cdot))\le e^{-Kt/2}\,|x-y|.
\end{equation}
Moreover, if $\sigma\sigma^*$ is positive then \eqref{W2gen} implies back (H.C.K.).\\
If we suppose moreover for some $m\ge2$
$$(H.C.K.m)\qquad\forall(x,y),\qquad \frac{m(m-1)}2|\sigma(x)-\sigma(y)|^2_{HS}+m\langle b(x)-b(y),x-y\rangle\le -K\,|x-y|^2$$
then
\begin{equation}\label{Wmgen}
W_m(P_t(x,\cdot),P_t(y,\cdot))\le e^{-Kt/m}\,|x-y|.
\end{equation}

\end{theorem}

\begin{proof}
The contraction in $W_2$ distance inherited from (H.C.K.) has already been proved. The contraction in $W_m$ distance is done exactly in the same way using once again synchronous coupling. The necessary part comes from \cite{BGG} (or more precisely section 4. in the Arxiv version 1110.3606). Let us explain the ideas of the proof. In fact, one may compute the time derivative of the Wasserstein distance: note $M:N=\sum_{i,j}M_{ij}N_{ij}$ when $M$ and $N$ are two matrices, then denoting $\nu_t$ and $\mu_t$ two solutions starting respectively from $\,u_0$ and $\mu_0$
$$\frac{d}{dt}W_2(\nu_t,\mu_t=2\,J(\nu_t|\mu_t)$$
where if $\nu_t=\nabla\phi_t\#\mu_t$,
\begin{eqnarray*}
&&J(\nu_t,|\mu_t)=\int\left[\frac12\sigma\sigma^*(x):(\nabla^2\phi_t(x)-I)+\frac12\sigma\sigma*(\nabla\phi_t(x)x):(\nabla^2\phi_t(x)^{-1}-I)\right.\\&&\qquad\qquad\qquad\qquad\qquad\qquad\qquad\qquad\qquad\qquad\left.-\langle b(\nabla\phi_t(x))-b(x),\nabla\phi_t(x)-x\rangle\right]d\mu_t.\end{eqnarray*}
Then the contraction property implies that at time 0 for $\nu_0=\delta_y$ and $\mu_0=\delta_x$
$$\frac K2 |x-y|^2\le J(\nu0\mu0).$$
A clever choice of $\phi$ then enables to prove the result.
\end{proof}

\medskip

Let $f$ be $C$-Lipschitz continuous. It holds $|f(X_t^x)-f(X_t^y)| \leq  C \, |X_t^x-X_t^y|$ so that, using \eqref{eqborne}, $P_tf$ is Lipschitz continuous with Lipschitz constant less than $C  \, e^{-Kt/2}$. 

As we said at the end of the previous section, when $K>0$ one deduces the existence and uniqueness of an invariant probability measure $\mu_\infty$, to which $\mu_T$ converges weakly.
\medskip

The rest of this (sub)subsection is devoted to give a proof of the following: if $f\in \mathcal A$ (see the introduction), then $(t,x) \mapsto P_tf(x)$ is regular and satisfies (for $t>0$) $$\partial_t P_tf = P_t L f = L P_t f \, .$$ 

{\em The reader who takes this result as granted can skip what follows.}

First if $f \in \mathcal A$, $Lf$ is $C_c^0$ and we have $P_tf(x) - f(x) = \int_0^t \, P_s(Lf)(x) \, ds$. It follows that $\lim_{s\to 0}\frac 1s \, (P_{t+s}f(x) - P_tf(x)) = P_t(Lf)(x)$ for all $x$, since $v \mapsto P_vLf(x)$ is continuous. So $\partial_t P_tf=P_t Lf$. The first delicate point is of course the commutation of $L$ and $P_t$. The second delicate point is the smoothness of $(t,x) \mapsto P_tf(x)$.

This commutation property is known if $\sigma$ and $b$ are in $C_b^\infty$ (see \cite{IW} p.254-258, boundedness of derivatives is important) in which case $(t,x) \mapsto P_tf(x)$ is actually $C^\infty$. But assuming boundedness of $b$ and its derivatives will exclude the cases of positive $K$.

We shall first show that $P_tf$ is a mild solution, provided $b$ does not grow too fast.
\begin{lemma}\label{lemsemi}
Assume that $|b(x)| \leq C \, (1+|x|^k)$ for some $C$ and $k \in \N$. Let $E$ be the space of continuous functions such that $x \mapsto f(x)/(1+|x|^k)$ is bounded, equipped with its natural norm $\parallel f\parallel = \sup_x \, (|f(x)|/(1+|x|^k))$. \\
Then $P_t$ is a continuous semi-group on $E$, any $f \in \mathcal A$ belongs to the infinitesimal generator of $P_t$ which coincides with $L$ on $\mathcal A$. Hence, $P_tf$ belongs to the domain of $L$ and $\partial_t P_tf = P_t L f = L P_t f$.
\end{lemma} 
\begin{proof}
It follows from lemma \ref{lemconcentration} that if $f\in E$, $P_tf$ is bounded by $c(t)\parallel f\parallel$, hence $P_t$ is a continuous semi-group on $E$ (whose range is included into $C_b$).\\
To see that $f\in \mathcal A$ belongs to the domain of the generator of $P_t$, we have to show that the convergence of $\lim_{s\to 0}\frac 1s \, (P_{s}f(x) - f(x))$ holds for the norm defined on $E$. But 
\begin{eqnarray*}
\left(\frac 1s \, (P_{s}f(x) - f(x))\right) - Lf(x) &=& \frac 1s \, \int_0^s (Lf(X_u^x)-Lf(x)) \, du\\ &\leq& \frac 1s \, \int_0^s \parallel \nabla Lf\parallel_\infty \, \E(|X_u^x - x|) \, du\\ &\leq& \frac 23 \, \parallel \nabla Lf\parallel_\infty \, (\alpha + \beta |b(x)|^2)^{\frac 12}  \, s^{\frac 12}
\end{eqnarray*}
according to \eqref{eqdev}. Since $|b(x)| \leq C \, (1+|x|^k)$, convergence holds for the norm on $E$. The proof is completed.
\end{proof}

If the coefficients are $C^\infty$, and $\partial_t -L$ is hypo-elliptic (for instance if $L$ is uniformly elliptic) it follows that $x \mapsto P_t f \in C^{\infty}$ and that the last equalities hold in the usual sense.

If we do not want to assume too much regularity on the coefficients, we have first to assume that $L$ is uniformly elliptic and call upon P.D.E. theory. If what follows is certainly well known by specialists, we include the argument.

Let $f \in \mathcal A$. For $k$ large enough, $B_k=B(0,k)$ contains the support of $f$. Consider the parabolic equation
\begin{eqnarray}\label{eqparab}
\partial_t u + Lu &=& 0 \quad \textrm{ in } (0,T)\times B_k \, \\ u(x,T) &=& f(x) \nonumber \\ u(x,t) &=& 0 \quad \textrm{ on } [0,T]\times \partial B_k \nonumber \, .
\end{eqnarray}
Since $f=0$ on $\partial B_k$ this makes sense. If $L$ is uniformly elliptic, it is known that there exists an unique solution $u_k$ in $C^{1,2}((0,T)\times B_k)$ of \eqref{eqparab}, and that this solution is represented as $$u_k(t,x) = \E(f(X^x_{T-t}) \, \BBone_{S^x_k>T}) = P_{T-t}f(x) - \E(f(X^x_{T-t}) \, \BBone_{S^x_k \leq T}) \, ,$$ where $S^x_k$ denotes the exit time from $B_k$ of $X_t^x$. For all this see \cite{Friedman}, in particular Theorem 5.2. p.147. \\ It follows in particular that for all $k$, $\parallel u_k\parallel_\infty \leq \parallel f\parallel_\infty$ and that for all $(t,x) \in (0,T)\times B_j$, $u_k(t,x) \to P_{T-t}f(x)$ as $k \to +\infty$ since $S_k^x \to +\infty$. \\ Now let $j$ be fixed, and look at $k > j$. The parabolic Schauder estimate tells us that there exists a constant $C_j$ depending on $j$, the ellipticity constant and the $C^2(B_{j+1})$ norms of $\sigma$ and $b$ such that $$\parallel u_k\parallel_{C^{2,\frac 12}(B_j)} \leq C_j \, \parallel u_k\parallel_\infty \leq C_j \, \parallel f\parallel_\infty \, ,$$ where $C^{k,\frac 12}$ is the set of $C^k$ functions with $\frac 12$-H\"{o}lder $k^{th}$ derivatives.\\ Arzela-Ascoli theorem tells us that a subsequence of $u_k$ converges in $C^{2}((0,T)\times B_j)$, and since the limit is $P_{T-t}f$, that the latter is $C^2$.
\bigskip

If $L$ is not uniformly elliptic, we may approximate it by $L_{\varepsilon} =L + \frac 12 \, \varepsilon \Delta$ for $\varepsilon \to 0$. If $\sigma \in C_b^2$, the diffusion matrix (field) $a_{\varepsilon} = \sigma_{\varepsilon} \, \sigma^*_{\varepsilon} + \varepsilon \, Id$ is $C_b^2$ and uniformly elliptic). It is known that its square root (in the sense of symmetric matrices) $a^{1/2}_{\varepsilon}$ is bounded and $C_b^2$ too. In addition $a^{1/2}_{\varepsilon} \to a^{1/2}$ as $\varepsilon \to 0$ the convergence taking place in $C_b^1$. In particular, (H.C.K($\varepsilon$)) holds for the new diffusion process with a constant $K(\varepsilon)$ going to $K$ as $\varepsilon \to 0$, provided (H.C.K) is satisfied for $\sigma = a^{\frac 12}$. 

It easily follows that, for all $t$, $\E(|X_t - X_t^\varepsilon|^2) \to 0$ as $\varepsilon \to 0$ provided $X_0=X_0^\varepsilon$. Indeed, Ito formula and Gronwall lemma again yield $$\E(|X_t - X_t^\varepsilon|^2) \leq t \, \sup_x |\sigma(x) - \sigma^\varepsilon(x)|_{HS}^2 \, \frac{e^{Dt}-1}{D} \, .$$  It follows that the convergence of time marginals holds in $W_2$ Wasserstein distance, hence in the weak topology.

The conclusion of the previous discussion is the following: if $a^{1/2}=\sigma \in C_b^2$, we may assume that $L$ is uniformly elliptic as far as the bounds we get do not depend on the ellipticity constant, and then go to the limit. We shall use this trick in the sequel.
\bigskip

\subsection{Commutation property with the gradient}

Once again, we will see how synchronous coupling or a contraction in Wasserstein distance provides the commutation property.

Let $f\in \mathcal A$, then $$f(X_t^x)-f(X_t^y) \leq \langle \nabla f(X_t^y),X_t^x-X_t^y \rangle + C |X_t^x-X_t^y|^2$$ for some constant $C$, so that 
\begin{eqnarray*}
|P_tf(x)-P_tf(y)| &=&  |\E(f(X_t^x)-f(X_t^y))| \\ &\leq& |\E(\langle \nabla f(X_t^y),X_t^x-X_t^y \rangle)| + C \E(|X_t^x-X_t^y|^2) \\ &\leq&  \left(\E(|\nabla f(X_t^y)|^2)\right)^{\frac 12} \, \left(\E(|X_t^x-X_t^y|^2)\right)^{\frac 12} + C \E(|X_t^x-X_t^y|^2) \\ &\leq& e^{-Kt/2} \, \left(P_t(|\nabla f|^2)(y)\right)^{\frac 12} \, |x-y| + C \, e^{-Kt} \, |x-y|^2 \, ,
\end{eqnarray*}
where the last step is done by using Theorem \ref{contW2gen}. Provided we know that $\nabla P_tf$ exists, we have thus obtained a weaker form of Proposition \ref{propgrad}, 

\begin{proposition}\label{propgrad2}
Assume (R) and (H.C.K) or the weaker contraction property \eqref{W2gen}. Let $f\in \mathcal A$. If $\nabla P_tf$ exists (which is true except possibly for (R5)), it holds $$|\nabla P_t f|^2 \leq e^{-Kt} \, P_t(|\nabla f|^2) \, .$$
\end{proposition}

Notice that, contrary to the Bakry-Emery bounded curvature case, the previous commutation property holds with the \emph{usual gradient} and not with the \emph{natural} one i.e. $\Gamma^{\frac12}$. 

If Proposition \ref{propgrad} allowed us to obtain logarithmic Sobolev inequalities, the weaker Proposition \ref{propgrad2} will allow us to obtain a weaker inequality, namely a Poincar\'e inequality.

\begin{remark}
It is worth mentioning here the following alternate proof of the commutation property, starting from Wasserstein contraction, as derived in the recent paper \cite{BaGL} following our suggestion, i.e. using Kantorovitch-Rubinstein duality we have for all bounded Lipschitz $\phi$ denoting the inf convolution operator $Q_t \phi(x)=\inf_y\{\phi(y)+\frac{|x-y|^2}{2t}\}$ and initial measure $\mu_0$ and $\nu_0$
\begin{eqnarray*}
\int Q_1\phi d\mu_t-\int \phi d\nu_t &=& \int P_tQ_1\phi d\mu_0-\int P_t \phi d\nu_0\\
&\le& e^{-Kt}W_2^2(\mu_0,\nu_0).
\end{eqnarray*}
Choose now $\mu_0=\delta_x$, $\nu_0=\delta_y$ to get for all $y$
$$P_t(Q_1\phi)(x)\le P_t\phi(y)+\frac{|x-y|^2}{2e^{Kt}}$$
which by homogeneity of the inf-convolution operator gives
$$P_t (Q_1\phi)\le Q_{e^{Kt}}(P_t\phi).$$
This assertion is in fact stronger than the gradient commutation property which can be deduced by using the fact that the inf-convolution operator is the Hopf-Lax solution of the Hamilton-Jacobi equation. \hfill $\diamondsuit$
\end{remark}

\subsection{h-processes and functional inequalities.}\label{subsechdiff}

We now introduce the corresponding $h$-process. Let $T>0$ and $h>0$ be such that $$\int \, P_Th \, d\mu_0 = 1 \, $$ We thus may define on the path-space up to time $T$ a new probability measure $$\frac{d\Q}{d\P_{\mu_0}}|_{\mathcal F_T} \, = \, h(\omega_T) \, .$$ Again $$\Q \circ \omega_s^{-1} = P_{T-s}h \, \mu_s \quad \textrm{ for all } \quad 0\leq s \leq T \, .$$ For simplicity, we assume in what follows that there exist $c$ and $C$ such that $C\geq h \geq c >0$. In this situation, using again Girsanov transform theory, we know that we can find a progressively measurable process $u_s$ such that $$\frac{d\Q}{d\P_{\mu_0}}|_{\mathcal F_T} \, = \, P_Th(\omega_0) \, \exp\left(\int_0^T \, \langle u_s  ,  dM_s\rangle \, - \, \frac 12 \, \int_0^T \, |\sigma(\omega_s) \, u_s|^2 \, ds\right) \, ,$$ where $\omega$ denotes the canonical element of the path-space and $M$ denotes the martingale part of $\omega$ under $\P_{\mu_0}$. In addition, it can be shown \cite{CL94} that 
\begin{equation}\label{eqentropgene}
H(\Q|\P_{\mu_0}) \, = \, H(h \mu_T|\mu_T) \, = \, H(P_Th \mu_0|\mu_0) + \frac 12 \, \E^\Q\left(\int_0^T \, |\sigma(\omega_s) \, u_s|^2 \, ds\right) \, .
\end{equation}
Again, at least formally it holds $$u_s = \nabla \, \log P_{T-s}h(\omega_s)$$ both $\P_{\mu_0}$ and $\Q$ almost surely. This is not only formal if $h\in \mathcal A$ and (R) (except (R5)) is satisfied. Assume both these conditions for the moment.

We thus have
\begin{equation}\label{eqentropgene1}
H(h \mu_T|\mu_T) \, = \, H(P_Th \mu_0|\mu_0) + \frac 12 \, \int_0^T \, \left(\int \, \frac{|\sigma \, \nabla P_s h|^2}{P_s h} \, d\mu_{T-s}\right) \, ds \, .
\end{equation}
Now define $$M=\parallel |\sigma|^2\parallel_{\infty}= \sup_{y} \, \sup_{|u|=1} \, |\sigma(y) u|^2 \, .$$ If $h \in \mathcal A$  we may apply Proposition \ref{propgrad2} in order to get (recall that $h\geq c$)
\begin{eqnarray}\label{eqlogsobgene}
H(h \mu_T|\mu_T) &\leq& H(P_Th \mu_0|\mu_0) + \frac M2 \, \int_0^T \, \left(\int \, e^{-Ks} \, \frac{ P_s(|\nabla h|^2)}{P_s h} \, d\mu_{T-s}\right) \, ds  \, \nonumber \\ &\leq& H(P_Th \mu_0|\mu_0) + \frac {M}{2c} \, \int_0^T \, e^{-Ks} \, \left(\int  \, |\nabla h|^2 \, d\mu_{T}\right) \, ds \nonumber \\ &\leq& H(P_Th \mu_0|\mu_0) +  \, \frac{M(1-e^{-KT})}{2cK} \, \int  \, |\nabla h|^2 \, d\mu_{T} \, ,
\end{eqnarray}
where we have used the Markov property for the second inequality.

Now let $g \in C_c^\infty$ be such that $\int g d\mu_T = \int P_T g \, d\mu_0 =0$ and choose $h=1+\eta g \in \mathcal A$ so that $\int P_Th d\mu_0=1$ and $h>c>0$ for $\eta$ small enough. Actually we will let $\eta$ go to $0$ so that in the limit $c=1$. Standard manipulations thus yield
\begin{equation}\label{eqpoincgene}
\int \, g^2 \, d\mu_T \leq \int \, (P_Tg)^2 \, d\mu_0 + \frac{M(1-e^{-KT})}{K} \, \int  \, |\nabla g|^2 \, d\mu_{T} \, .
\end{equation}

We can eventually use first the density of $C_c^\infty$ and then the trick we formerly described in order to relax the uniform ellipticity assumption (recall that $M=\sup_y \, \sup_{|u|=1}|\langle u,a(y)u\rangle|$ hence only depends on $a$ too).

Arguing as for Proposition \ref{proplogsob} we have obtained

\begin{proposition}\label{proppoincgene}
Assume that (R) and (H.C.K) are satisfied. Let $M=\parallel |\sigma|^2\parallel_{\infty}$. \\ If $\mu_0$ satisfies a Poincar\'e inequality with constant $C_P(0)$ then $\mu_T$ satisfies a Poincar\'e inequality with constant $$C_P(T)= e^{-KT} \, C_P(0) +\frac{M(1-e^{-KT})}{K} \, .$$ This applies in particular to $P(T,x,.)$ with $C_P(0)=0$.
\end{proposition}
\medskip

Contrary to the log-Sobolev inequality, the Poincar\'e inequality does not furnish a transportation inequality, so we shall try to adapt what we did in subsection \ref{subsecthp}.
\bigskip

\subsection{Transportation inequalities.}\label{subsecgene}
The situation is a little bit less simple than in the previous section. Indeed the martingale term is no more a brownian motion and we can no more use characterization tricks on $(\Omega,\Q)$. Hence 
we have to consider the solution of 
\begin{equation}\label{eqgirsgene}
dY_t = \sigma(Y_t) \, dB_t + b(Y_t) dt + a(Y_t) \, \nabla \log P_{T-t}h(Y_t) \, dt \, .
\end{equation}
As before we assume first that $h \in \mathcal A$, $C\geq h \geq c >0$ and that (R) is satisfied (except possibly (R5), so that \eqref{eqgirsgene} is well defined and admits a unique strong solution. We can thus build a solution with the same Brownian motion $B$ we used in \eqref{eqitogene}. Strong uniqueness follows from the local Lipschitz property of all the coefficients and non explosion (up to time $T$) which is ensured by construction ($\Q$ is a probability measure). Again we may choose in an appropriate way the distribution of the pair of initial variables.

If (H.C.K) is satisfied, it holds
\begin{eqnarray}\label{eqtransgene1}
\eta_t = \E(|Y_t-X_t|^2) &\leq& \eta_0 \, - \, K \, \int_0^t \, \eta_s \, ds + 2  \, \left(\int_0^t \E(\, \langle Y_s-X_s,a(Y_s) \, \nabla \log P_{T-s}h(Y_s)\rangle\right) \, ds \,  \nonumber \\ &\leq& \eta_0 \, - \, K \, \int_0^t \, \eta_s \, ds + \nonumber \\ & & + \, 2M^{\frac 12}  \, \left(\int_0^t \, \eta_s \, ds\right)^{\frac 12} \, \left(\E\left(\int_0^t \, |\sigma(Y_s) \nabla \log P_{T-s}h(Y_s) |^2 \, ds\right)\right)^{\frac 12} \nonumber \\ &\leq& \eta_0 \, - \, K \, \int_0^t \, \eta_s \, ds + 2 (2M)^{\frac 12}  \, \left(\int_0^t \, \eta_s \, ds\right)^{\frac 12} \, H^{\frac 12}(h\mu_T|\mu_T) \, .
\end{eqnarray}
We may thus conclude as in the previous section

\begin{proposition}\label{proptransgene}
Assume that (R) and (H.C.K) are satisfied. Let $M=\parallel |\sigma|^2\parallel_{\infty}$. \\ The conclusions of Proposition \ref{propdelta} and Proposition \ref{proptrans} are still true, replacing $C_T$ by $MC_T$\end{proposition}

Actually, when (R5) holds, we have proven this result for $h \in \mathcal A$ and $L+\frac \varepsilon 2 \, \Delta$. But as we have seen, $\mu_T(\varepsilon) \to \mu_T$ in $W_2$ distance, so that if $h$ is bounded the same holds for $z_\varepsilon \, h \mu_T(\varepsilon)$ to $h \mu_T$ ($z_\varepsilon$ being a normalization constant). Finally if a $T_2$ inequality holds for all $h \in \mathcal A$ it extends to all $h$ using density and the fact that $W_2(\nu,\mu) \leq \liminf W_2(\nu_n,\mu)$ if $\nu_n$ weakly converges to $\nu$.
\medskip

Of course a $T_2$ inequality implies a Poincar\'e inequality, but the constant in Proposition \ref{proppoincgene} is better (in addition we only require that $\mu_0$ satisfies a Poincar\'e inequality).

\begin{remark}\label{remconcent}
One of the renowned consequence of such inequalities is the concentration of measure phenomenon for $\mu_T$. In particular, under the assumptions of Proposition \ref{proptransgene}, $\mu_T$ satisfies a gaussian type concentration property. In particular $|X_T^x|^2$ has some exponential moment, fact we have already shown in lemma \ref{lemconcentration}. But this integrability does not reflect all the strength of the $T_2$ inequality whose tensorization property is particularly useful for statistical purposes. 

When $L$ is uniformly elliptic, this concentration property follows from gaussian estimates for the transition kernel. Here we obtain much more explicit constants (even if they are certainly far from optimality) which do not depend on the ellipticity constant. \hfill $\diamondsuit$
\end{remark}
\medskip

\begin{remark}\label{remelliptic}
Assume that $L$ is uniformly elliptic, i.e. $$e=\inf_y \, \inf_{|u|=1} \, |\sigma(y) u|^2 \, > \, 0 \, .$$ Then we deduce from Proposition \ref{proppoincgene} $$P_T g^2(x) - (P_T g(x))^2 \leq \frac {2M}{e} \ \frac{(1-e^{-KT})}{K} \, P_T(\Gamma g)(x) \, .$$ According to \cite{Ane} proposition 5.4.1, this is equivalent to the $CD(K/2,\infty)$ condition provided $M=e$ hence when $\sigma$ is constant times the identity. In the non constant diffusion case, our condition (H.C.K) seems to be really different from the Bakry-Emery curvature condition. \hfill $\diamondsuit$
\end{remark}
\medskip

\subsection{An hypoelliptic example : kinetic Fokker-Planck equation}\label{kFP}

We present in this section an application of the techniques developed here in an hypoelliptic example where the Bakry-Emery curvature is negative and where (H.C.K.) may not be satisfied also.\\
Let $(x_t,v_t)$ be the solution of the following SDE
\begin{eqnarray*}
dx_t&=&v_tdt\\
dv_t&=&dB_t-\nabla V(x_t)dt-v_tdt.
\end{eqnarray*}
also called stochastic Hamiltonian system. The long time behavior study of such a system has been considered for a long time and have been tackled by different techniques, see for example: hypocoercivity by Villani \cite{vil} or Lyapunov function technique by Bakry$\&$al \cite{BCG}. However, due to its hight degeneracy, the Bakry-Emery curvature is $-\infty$ so that we may not apply the $\Gamma_2$ technique. Remark also that the (H.C.K.) condition reads for all $(x,v)$ and $(y,w)$
$$-\langle\nabla V(x)-\nabla V(y),v-w\rangle-|v-w|^2\le - K(|x-y|^2+|v-w|^2)$$
so that it is hopeless to get $K>0$.\\

 Let us first remark that if $\nabla V$ is Lipshitz continuous, (H.C.K) is verified for some negative $K$ and using synchronous coupling, one may remark that we are in the same situation than in Section\ref{secdrift} so that  we get that for some negative $K$ the gradient commutation property holds
$$|\nabla P_t f|\le e^{-Kt}\,P_t|\nabla f|$$
and thus the logarithmic Sobolev inequality holds for $P_t((x,v),\cdot)$. Let us remark once again that those properties are written with the usual gradient and not the Carr\'e-du-Champ operator $\Gamma(f)=|\nabla_v f|^2$.\\
 
One may then wonder if it is possible to get the gradient commutation property with $K>0$. In fact, using synchronous coupling and It\^o's formula applied to the function $N((x,v),(y,w))=a|x-y|^2+b\langle x-y,v-w\rangle+|v-w|^2$, following \cite{BGM}, we get that if $V(x)=|x|^2+W(x)$ where $\nabla W$ is $\delta$-Lipschitz with $\delta$ sufficiently small there exists $a,b$ and $K>0$ such that $H$ is equivalent to the euclidean norm and
$$N((x_t^x,v_t^v),(x_t^y,v_t^w))\le e^{-Kt}N((x,v),(y,w))$$
so that we get as in Section \ref{secdrift} the commutation property for some $K>0$ and $A>1$
$$|\nabla P_tf|\le A\,e^{-Kt}\,P_t|\nabla f|$$
and thus a Logarithmic Sobolev inequality holds uniformly in time.\\

\medskip

It is not hard to extend the result of this simplified setting to the case where the Brownian motion in the velocity has a diffusion coefficient which is bounded and $L$-Lipschitz. We may then obtain a weaker gradient commutation property
$$|\nabla P_tf|^2\le A\,e^{-Kt}\,P_t|\nabla f|^2$$
and local Poincar\'e type inequality or Transportation information inequality like in Propositions \ref{proppoincgene} or \ref{proptransgene}, and if $L$ is sufficiently small uniform in time version of these inequalities (using functional $N$).

\subsection{Interpolation of the gradient commutation property and local Beckner inequality}

We have seen here that we cannot recover a logarithmic Sobolev inequality by our technique when (H.C.K.) is in force. Remember however that we have introduced the stronger (H.C.K.m) condition which implies a contraction in Wasserstein distance $W_m$. It is then not hard to deduce some interpolation of the gradient commutation property 
\begin{proposition}
Assume (R) and (H.C.K.m) or the weaker contraction property \eqref{Wmgen}. Let $f\in{\mathcal A}$, if $\nabla P_t f$ exists, it holds
\begin{equation}\label{interpcom}
|\nabla  P_tf|^{\frac{m}{m-1}}\le e^{-Kt/(m-1)}\,P_t\left(|\nabla f|^{\frac{m}{m-1}}\right).
\end{equation}
\end{proposition}
Remark once again that this property does hold even if the diffusion coefficient is degenerate, so that variations of the hypoelliptic example of the previous subsection with a diffusion coefficient in the velocity enters into this framework. This contraction property may thus lead to a reinforcement of the Poincar\'e inequality to a Beckner inequality.

\begin{proposition}
Assume (R) and (H.C.K.m) or the weaker \eqref{interpcom}. Let $M=\| |\sigma|^2\|_\infty$. Then for all nice $f$, we have the following Beckner inequality 
$$P_t f^2-P_t\left(|f|^{\frac{2m}{m+2}}\right)^{\frac{m+2}{m}}\le M \frac{m+2}{m}\frac{e^{2Kt/m}-1}{K}\, P_t|\nabla f|^2.$$
\end{proposition}

\begin{proof}
The proof unfortunately does not rely on the $h$-process introduced previously but on the $\Gamma_2$ type proof. Denote $p=\frac{2m}{m+2}$. By \eqref{interpcom} and H\"older inequality, for all nice non negative$f$
\begin{eqnarray*}
P_t f^2-(P_tf^p)^{2/p}&=&\int_0^t\frac{d}{ds}P_s(P_{t-s}f^p)^{2/p}\,ds\\
&\le&M\frac{2(2-p)}{p^2}\int_0^tP_s\frac{\nabla P_{t-s}f^p|^2}{(P_{t-s}f^p)^{2(p-1)/p}}ds\\
&\le&M\frac{2(2-p)}{p^2}\int_0^t e^{-2K(t-s)/m}\,P_s \frac{\left(P_{t-s}|\nabla f^p|^{\frac{m}{m-1}}\right)^{2(m-1)/m}}{(P_{t-s}f^p)^{2(p-1)/p}}ds\\
&\le&M\frac{2(2-p)}{p^2}\int_0^t e^{-2K(t-s)/m}\, P_t|\nabla f|^2ds\\
&=&M\frac{(2-p)m}{Kp^2}(e^{2Kt/m}-1)\, P_t|\nabla f|^2.
\end{eqnarray*}
\end{proof}


\subsection{Convergence to equilibrium in positive curvature.}\label{subsecconvgene}

Still in the uniform elliptic case, assume that $K>0$. We already mentioned that in this case $\mu_T$ weakly converges to the unique invariant probability measure $\mu_\infty$ (which exists).  In particular , for all smooth $g$ (say $C_b^2$), $\Var_{\mu_T}(g) \to \Var_{\mu_\infty}(g)$ as well as $\int |\nabla g|^2 \, d\mu_T \to \int |\nabla g|^2 \, d\mu_\infty$. We deduce that $$\Var_{\mu_{\infty}}(g) \leq \frac{M}{K} \, \int |\nabla g|^2 \, d\mu_{\infty} \leq \frac{2M}{eK} \, \int \Gamma(g) \, d\mu_{\infty} \, .$$ Summarizing all this we have obtained
\begin{theorem}\label{thmelliptic}
Assume that $\sigma$ is bounded and uniformly elliptic. Then if (H.C.K) holds for some $K>0$, defining $M$ and $e$ as before, there exists an unique invariant probability measure $\mu_\infty$ and $\mu_\infty$  satisfies a Poincar\'e inequality with constant $M/K$. In addition for all $f \in \L^2(\mu_\infty)$ it holds $$\Var_{\mu_\infty}(P_tf) \leq e^{-Ke T/M} \, \Var_{\mu_\infty}(f) \, .$$
\end{theorem}
As we said, this result is not captured by the $\Gamma_2$ theory. 

But we can obtain general convergence results, even in the non uniformly elliptic case. Indeed recall that in full generality $$\Var_{\mu_\infty}(P_t g) = \frac 12 \, \int_t^{+\infty} \int \, |\sigma \, \nabla P_sg|^2 \, d\mu_\infty \, ds \, .$$ Using proposition \ref{propgrad2} we thus have 
\begin{eqnarray*}
\Var_{\mu_\infty}(P_t g) &\leq&  \frac M2 \,  \int_t^{+\infty} \int \, |\nabla P_sg|^2 \, d\mu_\infty \, ds \\ &\leq& \frac M2 \, \int_t^{+\infty} e^{-Ks} \, \int \, P_s(|\nabla g|^2) \, d\mu_\infty \, ds \\ &\leq& \frac{M}{2K} \, e^{-Kt} \, \int \, |\nabla g|^2 \, d\mu_\infty \, . 
\end{eqnarray*}
Hence
\begin{theorem}\label{thmconvgene}
Assume that $\sigma$ is bounded. Then if (H.C.K) holds for some $K>0$, defining $M$ as before, there exists an unique invariant probability measure $\mu_\infty$ and for all nice enough function $g$ , $$\Var_{\mu_\infty}(P_t g) \leq  \, \frac{M}{2K} \, e^{-Kt} \, \int \, |\nabla g|^2 \, d\mu_\infty \, .$$ In addition if $\mu_\infty$ is symmetric (i.e. $\int fLg d\mu_\infty= \int gLf d\mu_\infty$), it holds $$\Var_{\mu_\infty}(P_t g) \leq  e^{-Kt}  \,  \Var_{\mu_\infty}(g) \, .$$
\end{theorem}
Remark once again that what is used here is the weak gradient commutation property which is a consequence of (H.C.K.). The last part of the theorem follows from a result in \cite{CGZ} recalled in the Appendix, lemma \ref{lemsymgeneral} (2). Of course, unless we explicitly know the invariant measure, it is not easy to see wether $\mu_\infty$ is symmetric or not.
\medskip

\begin{remark}\label{remDGW}
As said in the introduction, condition (H.C.K) already appears in \cite{DGW} (condition (4.5) therein). Assuming (H.C.K), these authors actually showed that the full law of the process up to time $T$ and starting from $x$, satisfies a $T_2$ transportation inequality w.r.t. the $\L^2$ metric on the path space (see \cite{DGW} theorem 5.6), at least when $K>0$ (also look at \cite{ustu} for similar results). The scheme of proof for transportation inequalities we developed here is similar (the only novelty is the use of the h-process in order to look at time marginals, while in \cite{DGW} the authors are using the transfer of transportation inequalities via Lipschitz mappings). We are also presumably more accurate with the assumptions required to build the coupling. \hfill $\diamondsuit$
\end{remark}
\medskip

We may further extends the previous argument to the entropic convergence to equilibrium. Let us suppose that there exists an unique invariant measure $\mu_\infty$.

\begin{theorem}
Assume that $\sigma$ is bounded and that the gradient commutation property
\begin{equation}\label{strongcom}|\nabla P_t f|\le c\,e^{-Kt}\, P_t|\nabla f|\end{equation}
holds for some positive $K$. Then for all nice positive function $f$ (defining $M$ as before)
$$\Ent_{\mu_\infty}(P_tf)\le \frac{cM}{K} e^{-Kt} \int\frac{|\nabla g|^2}{g}d\mu_{\infty}.$$
\end{theorem}

\begin{proof}
The proof is as for the $L_2$ decay quite standard. Indeed,
\begin{eqnarray*}
\Ent_{\mu_\infty}(P_tg)&\le&M\int_t^\infty \int\frac{|\nabla P_sg|^2}{P_s g}d\mu_\infty ds,\\
&\le&cM\int_t^\infty e^{-Ks}\int P_s\frac{\nabla g|^2}{g}d\mu_\infty ds,\\
&\le & \frac{cM}{K} e^{-Kt} \int\frac{|\nabla g|^2}{g}d\mu_{\infty}.
\end{eqnarray*}
\end{proof}

\begin{remark}
 One of the important point here is that we do not suppose any non-degeneracy on the diffusion coefficient, so that the result applies to the kinetic Fokker-Planck equation. It then provides an alternative to the approach by Villani \cite{vil}, where he obtained such kind of convergence by completely different techniques with assumptions quite similar to the ones described in Section \ref{kFP}. One may then complete the approach by regularization of the Fisher Information in small time to obtain an entropic decay controlled by the initial entropy, see \cite{vil} or \cite{GW}. \hfill $\diamondsuit$
\end{remark}
\begin{remark}
Let us point out that even in the symmetric case, such a control is not sufficient to recover a logarithmic Sobolev inequality as the analog of lemma \ref{lemsymgeneral} is no more valid for the entropy. Remark however that we have shown in section 2 how to recover a logarithmic Sobolev inequality for $P_t$ using  the strong commutation gradient property \eqref{strongcom}. If $K>0$, we may then let $t$ goes to infinity to recover a logarithmic Sobolev inequality for the invariant measure. It may be for example be used in the context of kinetic Fokker-Planck equation with non gradient coefficient, for which the invariant measure is unknown. \hfill $\diamondsuit$
\end{remark}
\begin{remark}
Let us consider, as in the Poincar\'e case via (H.C.K.) condition, a particular class of test function $g$ such that $g\ge\varepsilon>0$, so that $P_tg\ge \varepsilon$. We then see adapting the preceding proof that a weak commutation of gradient property
\begin{equation}\label{weakcom}|\nabla P_t f|^2\le c\,e^{-Kt}\, P_t|\nabla f|^2\end{equation}
obtained for example under (H.C.K.) condition implies that
$$\Ent_{\mu_\infty}(P_t g)\le \frac{cM}{\varepsilon K}\,e^{-Kt}\,\int |\nabla g|^2d\mu_\infty.$$
\end{remark}

\section{\bf Non homogeneous diffusion processes.}\label{secnonhom}

\subsection{General non homogeneous diffusion}

In \cite{CM} the authors extended the $\Gamma_2$ theory to time dependent coefficients (non homogeneous diffusions). Considering the Ito system
\begin{eqnarray}\label{eqitotime}
dX_t & = & \sigma(v_t,X_t) \, dB_t + b(v_t,X_t) \, dt \, .\\
dv_t & = & dt \, ,\nonumber \\
\mathcal L(v_0,X_0) & = & \delta_{t_0} \otimes \mu_0 \, , \nonumber
\end{eqnarray}
we see that all what we have done can be applied to this system. Actually one can modify the ``curvature'' assumptions introducing for some function $K(t)$ and its derivative $K'(t)$ :
\textbf{(H.C.K(t))} \quad for all $(x,y)$, all $t\in \R$  $$ |\sigma(t,x)-\sigma(t,y)|_{HS}^2 + 2 \, \langle b(t,x)-b(t,y),x-y\rangle \, \leq  \, - K'(t) \, |x-y|^2  \, .$$ We then have

\begin{theorem}\label{thmcurvet}
Assume that $\sigma$ and $b$ satisfy hypothesis (R) (considered as functions on $\R \times \R^n$). If (H.C.K(t)) is satisfied, then the conclusions of Theorem \ref{thmcurve} (1) (Poincar\'e) and (4)(log-Sobolev) still hold replacing $e^{-KT}$ by $e^{-K(T)}$ and $\frac{1-e^{-KT}}{K}$ by $\int_0^T \, e^{-K(s)} \, ds$. 
\end{theorem}
\begin{proof}
If $f$ only depends on $x$, the proof of proposition \ref{propgrad} (resp. \ref{propgrad2}) is unchanged using the process starting from $(0,x)$ and $(0,y)$ and replacing $Kt$ by $K(t)$. To obtain the analogue of proposition \ref{proplogsob} and proposition \ref{proppoincgene}, it suffices to remark that $\sigma \nabla$ is equal to $\nabla_x$, and use what precedes for $h$ depending on $x$ only. 
\end{proof}
For the transportation inequality we have to slightly modify the method in subsection \ref{subsecthp}. With the notations therein, \eqref{eqdist1} has become, $$\eta_t \leq \eta_0 \,  - \,  \int_0^t \, K'(s) \, \eta_s \, ds + 2 \sqrt 2 \,  H^{\frac 12}(h\mu_T|\mu_T) \, \left(\int_0^t \, \eta_s \, ds\right)^{\frac 12} \, ,$$ so that, as in the previous section we have to come back to
\begin{equation}\label{eqdist3t}
\eta_t  \leq   \eta_0 \,  - \, \int_0^t \, K'(s) \, \eta_s \, ds + 2 \, \int_0^t \, \E^{\Q}\left(|z_s-\omega_s| \, |\nabla \log P_{T-s}h(\omega_s)|\right) ds  \, .
\end{equation} 
Using as usual $(ab)^{\frac 12} \leq \lambda a + \frac 1\lambda b$ we obtain (see the details of the derivation in the previous section) that for all increasing function $\lambda(t)$  $$\eta'(t) \leq (-K'(t)+ \lambda'(t)) \, \eta_t  + \frac{4}{\lambda'(t)} I_T(h) \, ,$$ from which we deduce, provided we choose $K(0)=\lambda(0)=0$, $$\eta_T  \leq \, e^{-K(T)+\lambda(T)} \, \eta_0 \, + \, 4 \, e^{-K(T)+\lambda(T)} \, \left(\int_0^T \, \frac{e^{K(s)-\lambda(s)}}{\lambda'(s)} \, ds\right) \, I_T(h) \, .$$ 
\begin{theorem}\label{thmcurvettrans}
Assume that $\sigma$ and $b$ satisfy hypothesis (R) (considered as functions on $\R \times \R^n$). If (H.C.K(t)) is satisfied, then for any $x$ and any increasing function $\lambda$, $P(T,x,.)$ satisfies a $W_2I$ inequality $$W_2^2(h P(T,x,.),P(T,x,.)) \leq C(T) \, \int \, \frac{|\nabla h|^2}{h} \, d\mu_T \, ,$$ with constant $$C(T) \leq  4 \, e^{-K(T)+\lambda(T)} \, \left(\int_0^T \, \frac{e^{K(s)-\lambda(s)}}{\lambda'(s)} \, ds\right) \, .$$ If $\mu_0$ satisfies a $T_2$ inequality with constant $C_T(0)$, then $$W_2^2(h\mu_T,\mu_T) \leq C_T(0) \,e^{-K(T)+\lambda(T)} \,  H(h\mu_T|\mu_T) + C(T) \, I_T(h) \, .$$ \end{theorem}
The best choice of $\lambda$ is not clear. If $K'(t)$ is not positive on the whole $[0,T]$, it seams that taking $\lambda(T)=\lambda T$ for some $\lambda>0$ is enough. If $K'(t)>0$ for all $t$ (but not necessarily bounded from below by a positive constant), $\lambda(t)= \lambda K(t)$ seems to be natural.

\medskip
\begin{remark}\label{remconverge}
Assume that $K(t) \to +\infty$ as $t \to +\infty$ and that $C=\int_0^{+\infty} e^{-K(s)} \, ds < +\infty$. Then, for all $t$, $P(t,x,.)$ (the distribution of the process starting from $x$ and $t_0=0$) satisfies a Poincar\'e inequality (and a log-Sobolev inequality when $\sigma=Id$) with a constant bounded by $MC$ (or $2C$). The family $(P(t,x,.))_{t>0}$ is then tight, but we do not know whether it is weakly convergent or not. Nevertheless any weak limit satisfies the same functional inequality.\\ When $\sigma=Id$ we know that $|X_t^x-X_t|\leq e^{-K(t)} |x-X_0|$ for any initial random variable $X_0$. It follows that if a sequence $P(t_k,x,.)$ is weakly convergent to some $\mu$, the sequence $\mu_{t_k}$ weakly converges to the same limit. \\ In particular if we consider $\sigma=Id$, $b(t,x)=-\frac 12 \, (\nabla U(x) + K'(t) \, x)$, for some convex potential $U$, (H.C.K(t)) is satisfied, so that any weak limit satisfies a log-Sobolev inequality. If $d\mu =e^{-U} dx$ does not satisfy a log-Sobolev inequality, it cannot be a weak limit, even if $K'(t) \to 0$. In this situation one should expect that the ``perturbation'' of $\nabla U$ being smaller and smaller when $t$ growths, the convergence to $\mu$ will still hold. This is not the case. \hfill $\diamondsuit$
\end{remark}

\subsection{Application to some non-linear diffusions.}\label{exnonlinear}
We shall now discuss an example that does partly enter the framework of the beginning of this section.

Following \cite{malrieu01,CGM} consider the following non-linear stochastic differential equation
\begin{eqnarray}\label{eqgranul}
dX_t &=& dB_t - \frac 12 \, \nabla V(X_t) \, dt - \, \frac 12 \nabla W*q_t(X_t) \, dt \, \\
\mathcal L(X_t) &=& q_t \, dx \nonumber \, .
\end{eqnarray}
If a solution exists, $q_t$ will solve 
\begin{equation}\label{eqgranul2}
\partial_t q_t = \frac 12 \,\nabla. \left(\nabla q_t + q_t\nabla V + q_t(\nabla W*q_t)\right) \, .
\end{equation}
 This is a non-linear diffusion of Mc Kean-Vlasov type modeling, for instance, granular media. We refer to the introduction of \cite{CGM} for details and motivations. One can approximate the solution of \eqref{eqgranul} by the first coordinate of a linear large particle system with mean field interactions. This is what is done in \cite{malrieu01,CGM} to study the long time behavior of $X_t$. 

Let we see how to apply what we have just done. First, under some conditions on $V$ and $W$ (we later shall give some of them) existence and weak uniqueness of \eqref{eqgranul} are ensured, provided the initial law admits some big enough polynomial moment. This will imply for all $x$, the existence and uniqueness of $q_t^x$ solution of \eqref{eqgranul2} with initial condition $\delta_x$. As usual for these non-linear equations, if we consider the linear time inhomogeneous S.D.E. $$dZ^{x,y}_t = dB_t - \frac 12 \, \nabla V(Z^{x,y}_t) \, dt - \, \frac 12 \, \nabla W*q^x_t(Z^{x,y}_t) \, dt \quad Z^{x,y}_0=y \, ,$$ the pathwise unique solution (up to explosion) $Z_.^{x,x}$ is shown to satisfy \eqref{eqgranul} (i.e. $\mathcal L(Z^{x,x}_t)=q_t^x$) so that it coincides with $X^x_t$. So, once $q_t^x$ and $q_t^y$ are built, we may build our synchronous coupling $(X_t^x,X_t^y)$ as before. Now introduce an independent copy $(\bar X_t^x,\bar X_t^y)$ of $(X_t^x,X_t^y)$.

We have $$ \mathbb E \left(|X_t^x - X_t^y|^2\right) =$$
\begin{eqnarray}\label{eqnonlinecart}
&=& \,  - \, \E \left(\int_0^t \, \langle \nabla V(X_s^x)- \nabla V(X_s^y),X_s^x-X_s^y\rangle ds\right)  \\ & & - \, \E\left(\int_0^t \, \int \, \langle \nabla W(X_s^x-z^x)- \nabla W(X_s^y-z^y),X_s^x-X_s^y\rangle \, q_s^x(z^x) \, q_s^y(z^y) \, dz^x \, dz^y \, ds\right) \, . \nonumber 
\end{eqnarray}
Remark that the last term can be written $$\int_0^t \, \E \left(\langle \nabla W(X_s^x-\bar X_s^x)- \nabla W(X_s^y-\bar X_s^y),X_s^x-X_s^y\rangle\right) \, ds \, .$$ If we assume in addition (as usual) that $W(-x)=W(x)$, and remember that $\bar X$ is a copy of $X$, it is still equal to $$ - \, \int_0^t \, \E \left(\langle \nabla W(X_s^x-\bar X_s^x)- \nabla W(X_s^y-\bar X_s^y),\bar X_s^x- \bar X_s^y\rangle\right) \, ds \, .$$ Hence
\begin{eqnarray*}
2 \, \mathbb E \left(|X_t^x - X_t^y|^2\right) && = \mathbb E \left(|X_t^x - X_t^y|^2\right) + \mathbb E \left(|\bar X_t^x - \bar X_t^y|^2\right)\\ && = \, 2 |x-y|^2 \,  - \, 2 \, \E \left(\int_0^t \, \langle \nabla V(X_s^x)- \nabla V(X_s^y),X_s^x-X_s^y\rangle ds\right) \\ && - \, \int_0^t \, \E \left(\langle \nabla W(X_s^x-\bar X_s^x)- \nabla W(X_s^y-\bar X_s^y),(X_s^x-\bar X_s^x)-(X_s^y-\bar X_s^y)\rangle\right) \, ds \, .
\end{eqnarray*}
We may thus state
\begin{theorem}\label{thmgranul}
Assume that 
\begin{enumerate}
\item[H1] \quad $V$, $W$ and their first two derivatives have at most polynomial growth of order $m$ and  $W(-x)=W(x)$,
\item [H2] \quad $V$ satisfies (H.C.$K_V$) and $W$ satisfies (H.C.$K_W$).
\end{enumerate}
Let $a=\max(m(m+3),2m^2)$. If $\mu_0$ and $\nu_0$ have a polynomial moment of order $a$, there exist an unique solution of \eqref{eqgranul} and an unique solution of \eqref{eqgranul2} among the set of probability flows having having a polynomial moment of order $a$ with initial condition $\mu_0$ or $\nu_0$.\\ Furthermore
\begin{enumerate}
\item $$W_2^2(\mu_T,\nu_T)=W_2^2(q^{\mu_0}_T \, dx,q^{\nu_0}_T \, dx) \leq e^{-(K_V+\min(K_W,0))T} \, W_2^2(\mu_0,\nu_0) \, .$$
\item If $V=0$ and $\int x \mu_0(dx) = \int x \nu_0(dx)$ then $$W_2^2(\mu_T,\nu_T) \leq e^{-K_W T} \, W_2^2(\mu_0,\nu_0) \, .$$
\end{enumerate}
Introduce the conditions,
\begin{enumerate}
\item[H'1] \quad  $K=K_V+\min(K_W,0)>0$.
\item[H'2] \quad  $V=0$, $\int x \mu_0(dx) = \int x \nu_0(dx)$ and $K_W>0$.
\end{enumerate}
If H'1 is satisfied, there exists an unique invariant distribution $\mu_\infty=q^\infty(x) dx$ of \eqref{eqgranul} and \eqref{eqgranul2} satisfying the polynomial moment condition of order $a$, the convergence to $\mu_\infty$ in $W_2$ Wasserstein distance being exponential as above.\\ If H'2 is satisfied the same result holds for each $A\in \R^n$ in the set of probability measures such that $\int x \mu(dx)=A$.
\end{theorem}
\begin{proof}
The moment condition ensuring existence and uniqueness is described in \cite{CGM} section 2.

According to what precedes we have 
\begin{eqnarray*}
\mathbb E \left(|X_t^x - X_t^y|^2\right) &\leq& |x-y|^2 - K_V \int_0^t \, \mathbb E \left(|X_s^x - X_s^y|^2\right) ds\\ &&  - \, (K_W/2) \, \int_0^t \, \mathbb E \left(|(X_s^x - \bar X_s^x) - (X_s^y-\bar X_s^y)|^2\right) ds \, \\ &\leq& |x-y|^2 - K_V \int_0^t \, \mathbb E \left(|X_s^x - X_s^y|^2\right) ds \\ && - \, K_W \, \left(\int_0^t \, \mathbb E \left(|X_s^x - X_s^y|^2\right) ds - \int_0^t \, |\mathbb E \left(X_s^x - X_s^y\right)|^2 ds \right) \, .
\end{eqnarray*}
Of course we may replace the initial $\delta_x$ and $\delta_y$ by probability distributions $\mu_0$ and $\nu_0$ satisfying the required moment conditions. This immediately furnishes the first assertion about the upper bound for the Wasserstein distance.

If $V=0$ it is easily seen that $\int x q_t^{\mu_0}(x) dx = \int x \mu_0(dx)$ for all $t>0$, hence $$E \left(X_s^{\mu_0} - X_s^{\nu_0}\right)=0$$ provided the same holds at time $0$. This furnishes the second assertion for the upper bound.

Finally the convergence under strict positivity of our new ``curvature'' condition ensures the existence of the limiting measure $\mu_\infty$. To see that $\mu_\infty=q^\infty(x) dx$ is actually invariant, one can for instance use the following trick: first consider the solution $q_t^\infty$ of \eqref{eqgranul2} with initial condition $q^\infty$. Similar bounds for the Markov non homogeneous process $Z_.^{q^\infty,y}$ (when we replace $q_t^x$ by $q_t^\infty$) are obtained applying the results of the beginning of this section. Hence the law of $Z_T^{q^\infty,q^\infty}$ (which is exactly $\mu_T$ starting with $\mu_\infty$ as we explained before) converges to some limiting measure $\mu_\infty^{\mu_\infty}$ which in turn is equal to $\mu_\infty$ and is invariant for $Z_.^{q_\infty,y}$. This achieves the proof. 
\end{proof}
\begin{remark}\label{remmalrieu}
The proof of the above result is new and direct, while the result is mainly contained in \cite{malrieu01,CGM} using particle approximation. Notice that in \cite{malrieu01} the $\Gamma_2$ approach is developed for the non homogeneous Markov diffusion $Z_.$ and not for $X_.$. Also notice that some direct study of the decay to equilibrium in $W_2$ distance for granular media is done in \cite{BGGgran}. \hfill $\diamondsuit$
\end{remark}
As said in the previous remark the $\Gamma_2$ theory does not work directly for the process $X_.$. Actually our method to control the gradient of $x \mapsto \E(f(X_t^x))$ should work but we do not know whether the gradient exists or not, due to the fact that we do not have any a priori regularity in the initial condition. Fortunately, if we want to obtain some properties for the time marginal distribution $\mu_T$ we may use the fact (as done by Malrieu) that this distribution coincides with the one of the non homogeneous Markov diffusion $Z_T^{x,y}$ to which we can apply the techniques of this section. In particular, in the situation of the previous theorem, when $q^\infty$ exists we may consider the diffusion $$dZ^{y}_t = dB_t - \frac 12 \, \nabla V(Z^{y}_t) \, dt - \, \frac 12 \, \nabla W*q^\infty(Z^{y}_t) \, dt \quad Z^{y}_0=y \, ,$$ for which $\mu_\infty(dx)=q^\infty(x) dx$ the invariant probability measure. Using the results in section \ref{secdrift} we thus have
\begin{proposition}\label{proplsmckean}
In the situation of Theorem \ref{thmgranul}, if H'1 or H'2 are satisfied, $\mu_\infty$ satisfies a log-Sobolev inequality with constant $C_{LS}=2/K$ or $C_{LS}=2/K_W$.
\end{proposition}

All what we have done extends to more general Mc Kean-Vlasov equations, with a diffusion coefficient $\sigma$ and a drift $b$ satisfying hypothesis (R). In particular, positive curvature (in the sense of (H.C.K)) will also imply existence of and convergence to an invariant probability measure. The only difference is that we have to replace log-Sobolev inequality by Poincar\'e inequality in the latter proposition. Let us explain quickly what kind of model we may consider. We do not aim to be optimal, but will provide a flavor of the results on contraction with some non constant diffusion term. We will not focus also on the existence of solution of such equation. Let $X_t^x$ be solution of 
\begin{eqnarray}\label{eqgranul3}
dX_t^x &=& \sigma(X_t^x,\kappa*q_t(X_t^x))dB_t - \frac 12 \, \nabla V(X_t^x) \, dt - \, \frac 12 \nabla W*q_t^x(X_t^x) \, dt \, \\
X^x_0&=&x\\
\mathcal L(X_t^x) &=& q_t^x \, dx \nonumber \, .
\end{eqnarray}
\begin{theorem}
Let us suppose H1 and H2, that $\kappa$ is $l$-Lipschitz and that 
$$|\sigma(x,y)-\sigma(x',y')|^2_{HS}\le r(|x-x'|^2+|y-y'|^2).$$
Then (using the notation of Th.\ref{thmgranul})
$$W_2^2(\mu_T,\nu_T)\le e^{-(K_V-r(1+4l^2)+\min(K_W,0))T}\,W_2^2(\mu_0,\nu_0).$$
Suppose moreover that $K_V-r(1+4l^2)+\min(K_W,0)>0$, then there exists an unique invariant distribution to \eqref{eqgranul3}, the convergence to $\mu_\infty$ in $W_2$ Wasserstein distance being exponential as above.
\end{theorem}

The proof follows the same line as before except that in the It\^o's formula, there is the diffusion part which comes into play for which we use the Lipschitz condition of the theorem. Note that Bolley$\&$al \cite{BGM} have considered the case of a kinetic McKean-Vlasov equation, but with a constant diffusion coefficient in speed. As before, we may obtain some functional inequality for the invariant distribution as in Prop. \ref{proplsmckean} but we have to replace log-Sobolev inequality by Poincar\'e inequality .

\bigskip

\section{\bf Extensions to some non uniformly convex potentials.}\label{secext}

Let us come back to \eqref{eqito}, and assume that $\Upsilon$ is bounded. We shall extend (H.C.K) to more general situations. The first natural extension is to replace the squared distance by some other convex functional of the distance. More precisely.

\begin{definition}\label{defphi}
Let $\varphi : \R^+ \to \R^+$. We say that $\varphi$ belongs to $\CC$ if it satisfies the following conditions: 
\begin{itemize}
\item $\varphi$ is increasing and convex, with $\varphi(0)=0$ and $\varphi(1)=1$, \item $a \mapsto \varphi(a)/a$ is non decreasing, \item there exist a positive function $\psi$ such that for all $a>0$ and all $\lambda>0$, $\varphi^{-1}(\lambda a) \leq \psi(\lambda) \, \varphi^{-1}(a)$, where  
$\varphi^{-1}$ denotes the inverse (reciprocal) function of $\varphi$.
\end{itemize}  
\end{definition}

\begin{definition}\label{defHphi}
Let $\varphi \in \CC$. We shall say that \textbf{(H.$\varphi$.K)} is satisfied for some $K>0$ if for all $(x,y)$, $$\langle \nabla U(x) - \nabla U(y), x-y \rangle \geq K \, \varphi(|x-y|^2) \, .$$
\end{definition}
On one hand, since $K>0$ and $\varphi \geq 0$, (H.$\varphi$.K) implies that $U$ is convex. On the other hand, if (H.$\varphi$.K) is satisfied, since $U$ is smooth, $\varphi(a)/a$ is necessarily bounded near the origin since $\limsup_{a \to 0} (\varphi(a)/a) \leq \inf |Hess(U)|$. Here of course if $\varphi \in \CC$ the latter is automatically satisfied.

If $\varphi(a)=a$ this is nothing else but (H.C.K).  If $\varphi(a)/a \to +\infty$ we shall say that $U$ is super-convex. This terminology is justified by the example below.
\smallskip

\begin{example}\label{exsuper}
Let $U(x)=(|x|^2)^{\beta}$ for some $\beta>1$. We shall see that (H.$\varphi$.K) is satisfied for $\varphi(a)=a^{\beta}$ and some $K$ we shall estimate.

We start with the one dimensional case. In this case $$(U'(x)-U'(y))(x-y) = 2\beta \, (sign(x) |x|^{2\beta-1} - sign(y) |y|^{2\beta-1})(x-y) \, .$$ If $sign(x)=sign(y)$, we may assume that $|x|\geq |y|$, write $|x|=u+|y|$ for $u\geq 0$ and remark that if $2\beta-1\geq 1$, $$(u+|y|)^{2\beta-1} - |y|^{2\beta-1} \geq u^{2\beta-1}$$
so that $$(U'(x)-U'(y))(x-y) = 2\beta \, ((u+|y|)^{2\beta-1} - |y|^{2\beta-1})u \geq 2 \beta u^{2\beta} = 2 \beta |x-y|^{2\beta} \, .$$ If $sign(x)=-sign(y)$, we have, using the convexity of $x \mapsto |x|^{2\beta-1}$, $$(U'(x)-U'(y))(x-y) = 2\beta \, (|x|^{2\beta -1}+|y|^{2\beta -1})(|x|+|y|) \geq 2\beta \, 2^{2-2\beta} \, (|x|+|y|)^{2\beta} = 2\beta \, 2^{2-2\beta} \, |x-y|^{2\beta} \, .$$ Since $\beta > 1$, we may choose $K_\beta=2\beta \, 2^{2-2\beta}$.
\smallskip

 The general situation is a little bit more intricate.\\ Pick $x$ and $y$ in $\R^n$, assume that $|x|\geq |y|$ and write $x=|x|u$ and $y=|y|(\alpha u + \gamma v)$ for unit vectors $u$ and $v$ such that $\langle u,v\rangle =0$ and $\alpha^2+\gamma^2=1$. Then $$\langle \nabla U(x) - \nabla U(y),x-y\rangle = 2\beta  \left((|x|^{2\beta-1} - \alpha |y|^{2\beta-1})(|x|-\alpha |y|) + \gamma^2 |y|^{2\beta}\right) \, ,$$ and $$|x-y|^{2\beta} = \left((|x|-\alpha |y|)^2 + \gamma^2 |y|^2\right)^{\beta} \, \leq \, 2^{\beta-1} \left((|x|-\alpha|y|)^{2\beta} + \gamma^{2\beta} |y|^{2\beta}\right).$$
 \smallskip
 
 If $\alpha \geq 0$, we write again $|x|=|y|+a$ with $a\geq 0$. Thus , since $0\leq 1-\alpha\leq 1$,
\begin{eqnarray*} 
 |x|^{2\beta-1} - \alpha |y|^{2\beta-1} &=& (a+|y|)^{2\beta-1} - \alpha|y|^{2\beta-1}\geq a^{2\beta-1} + (1-\alpha)|y|^{2\beta-1} \\ &\geq&  a^{2\beta-1} + ((1-\alpha)|y|)^{2\beta-1}\\ &\geq& 2^{2-2\beta} \, (a+(1-\alpha)|y|)^{2\beta -1} = 2^{2-2\beta} \, (|x|-\alpha|y|)^{2\beta -1} \, .
 \end{eqnarray*}
It follows, since $\beta \geq 1$ and $\gamma^2\leq 1$,
\begin{eqnarray*}
\langle \nabla U(x) - \nabla U(y),x-y\rangle &\geq& 2\beta \, \left(2^{2-2\beta} \, (|x|-\alpha|y|)^{2\beta}+\gamma^2 \, |y|^{2\beta}\right)\\ &\geq& 2\beta \, 2^{2-2\beta} \, \left((|x|-\alpha|y|)^{2\beta}+\gamma^{2\beta} \, |y|^{2\beta}\right)\\ &\geq& 2\beta \, 2^{3-3\beta} \, |x-y|^{2\beta} \, .
\end{eqnarray*} 
\smallskip

If $\alpha <0$, since $|\alpha|\leq 1$, it holds 
\begin{eqnarray*}
\langle \nabla U(x) - \nabla U(y),x-y\rangle &=& 2\beta \, \left((|x|^{2\beta-1}+|\alpha||y|^{2\beta-1})(|x|+|\alpha||y|)+\gamma^2 \, |y|^{2\beta}\right)\\ &\geq& 2\beta \, \left((|x|^{2\beta-1}+(|\alpha||y|)^{2\beta-1})(|x|+|\alpha||y|)+\gamma^{2\beta} \, |y|^{2\beta}\right)\\ &\geq& 2\beta \, \left(2^{2-2\beta}(|x|+|\alpha||y|)^{2\beta}+\gamma^{2\beta} \, |y|^{2\beta}\right)\\&\geq& 2\beta \, 2^{3-3\beta} \, |x-y|^{2\beta} \, .
\end{eqnarray*}
\smallskip

\begin{proposition}\label{propsuper}
Let $U(x)=(|x|^2)^{\beta}$ for some $\beta>1$. Then (H.$\varphi$.K) is satisfied for $\varphi(a)=a^{\beta}$ and $K_\beta \geq 2\beta \, 2^{3-3\beta}$. If $n=1$ we have the better bound $K_\beta \geq 2\beta \, 2^{2-2\beta}$. \hfill $\diamondsuit$
\end{proposition}
\end{example}

\begin{remark}\label{remcgm}
If $\varphi \in \CC$, for all $a\geq 0$ and all $\varepsilon >0$, it holds $$\varphi(a) \geq \frac{\varphi(\varepsilon)}{\varepsilon} \, a \, - \, \varphi(\varepsilon) \, .$$ Hence (H.$\varphi$.K) implies the following condition
\begin{equation}\label{eqcgm}
\textrm{ for all $\varepsilon >0$ and all $(x,y)$} \quad \langle \nabla U(x) - \nabla U(y), x-y \rangle \geq K \, \left(\frac{\varphi(\varepsilon)}{\varepsilon} \, |x-y|^2 \, -  \, \varphi(\varepsilon)\right) \, .
\end{equation}
The latter appears in the study of the granular medium equation in \cite{CGM} (condition(6)) for power functions $\varphi$. This formulation will be the interesting one. It can be extended in 
\begin{definition}\label{defHalpha}
Let $\alpha$ be a non decreasing function defined on $\R^+$. We shall say that \textbf{(H.$\alpha$.K)} is satisfied for some $K>0$ if for all $(x,y)$ and all $\varepsilon >0$, $$\langle \nabla U(x) - \nabla U(y), x-y \rangle \geq K \, \alpha(\varepsilon) \, (|x-y|^2 - \varepsilon) \, .$$ (H.$\varphi$.K) implies (H.$\alpha$.K) with the same $K$ and $\alpha(\varepsilon)= \varphi(\varepsilon)/\varepsilon$. In this definition we do not need that $a \mapsto a \alpha(a)$ is convex.
\end{definition}
\end{remark}
Now we shall see how to use (H.$\varphi$.K).
\medskip

\subsection{Non fully convincing first results.}\label{subsecsuperno}
 This subsection contains first results which are not really convincing, but have to be tested. \\ 
 If we want to control the gradient $\nabla P_tf$, we may write for $t>u$, 
\begin{eqnarray*}
|X_t^x-X_t^y|^2 &=&|X_u^x-X_u^y|^2 - \, \int_u^t \, \langle \nabla U(X_s^x) - \nabla U(X_s^y),X_s^x-X_s^y\rangle \,  \, ds \\ &\leq& |X_u^x-X_u^y|^2 \, - \, K  \, \int_u^t \, \varphi(|X_s^x-X_s^y|^2) \, ds \, .
\end{eqnarray*}
Denoting $\eta_t=|X_t^x-X_t^y|^2$, we thus have $\eta'_t \leq -K \, \varphi(\eta_t)$. If $\varphi(a)=a^\beta$, this yields
\begin{equation}\label{eqdecsuper}
|X_t^x-X_t^y|^2 \leq |x-y|^2 \, \left(\frac{1}{1+K (\beta-1)|x-y|^{2(\beta-1)} \, t}\right)^{1/(\beta-1)} \, .
\end{equation} 

This result (even after taking expectation) is not really satisfactory. Indeed, first we do not obtain any better control for $\nabla P_tf$ than the one for a general convex potential (in particular we do not obtain a rate of convergence to $0$). In second place, the decay to $0$ of the Wasserstein distance we obtain is desperately slow, while we expected an exponential decay (which we know to hold true for $U(x)=|x|^{2\beta}$ for $\beta \geq 1$). Notice however that we recover the exponential decay we obtained previously when $\beta \to 1$.

\begin{remark}\label{remalacgm}
If instead of (H.$\varphi$.K) we use (H.$\alpha$.K), it is not difficult to show that $$\eta_t \leq \eta_0 \, e^{-K \alpha(\varepsilon)t} + \varepsilon \, .$$ If $\alpha(\varepsilon)= \varepsilon^{\beta-1}$, choosing $\varepsilon=\eta_0 \, t^{-\theta}$ for some $\theta<\beta-1$, we get $$W_2^2(P(t,x,.),P(t,y,.)) \leq |x-y|^2 \, (t^{-\theta} + e^{-K |x-y| t^{\beta-1-\theta}}) \, .$$ The method can be extended to the Mc Kean-Vlasov situation studied in subsection \ref{exnonlinear} and allows us to recover (up to the constants) Theorem 4.1 in \cite{CGM} without the help of a particle approximation. However, better results in this situation are obtained in \cite{BGGgran}. \hfill $\diamondsuit$
\end{remark}

Mimicking subsection \ref{subsecthp}, in particular \eqref{eqdist1}, do we obtain more interesting results ? Using the notation therein we have 
\begin{equation}\label{eqdist1super}
\eta_t := \E^{\Q}(|z_t-\omega_t|^2)  \leq   \eta_0 \,  - \, K \, \int_0^t \, \varphi(\eta_s) \, ds + 2 \sqrt 2 \,  H^{\frac 12}(h\mu_T|\mu_T) \, \left(\int_0^t \, \eta_s \, ds\right)^{\frac 12} \, . 
\end{equation}
Using Jensen inequality we deduce
\begin{eqnarray*}
\varphi\left(\frac 1t \, \int_0^t \, \eta_s \, ds\right)&\leq& \frac 1t \, \int_0^t \, \varphi(\eta_s) \, ds \\ &\leq& \frac{2 \sqrt 2}{K \, t} \, H^{\frac 12}(h\mu_T|\mu_T) \, \left(\int_0^t \, \eta_s \, ds\right)^{\frac 12}
\end{eqnarray*}
so that, if $v_t=\int_0^t \, \eta_s \, ds$,
\begin{eqnarray*} 
v_t &\leq& t \, \varphi^{-1}\left(\frac{2 \sqrt 2}{K \, t} \, H^{\frac 12}(h\mu_T|\mu_T) \, v_t^{\frac 12}\right)\\ &\leq& t \, \psi\left(\frac{2 \sqrt 2}{K \, t} \, H^{\frac 12}(h\mu_T|\mu_T)\right) \, \varphi^{-1}(v_t^{1/2}) \, .
\end{eqnarray*}
If $\varphi(a)=a^\beta$, we thus obtain
\begin{eqnarray}\label{eqdist2super}
\eta_T & \leq & \eta_0 + 2 \sqrt 2 \,  H^{\frac 12}(h\mu_T|\mu_T) \, \left(\int_0^T \, \eta_s \, ds\right)^{\frac 12} \, \nonumber \\  & \leq & \eta_0 + (2 \sqrt 2)^{\frac{\beta+1}{\beta}} \, K^{-1/\beta} \, H^{\frac {\beta+1}{2\beta}}(h\mu_T|\mu_T) \, T^{\frac{\beta-1}{2\beta-1}} \, .
\end{eqnarray}
This result is certainly not fully satisfactory too. On one hand, we get a less explosive bound in time (recall that in the general convex case the bound growths like $T$), but on the other hand the relative entropy appears  to a power less than 1. In particular such an inequality does not imply a Poincar\'e inequality (which is obtained for entropies going to $0$), but furnishes nice concentration properties (obtained for large entropies via Marton's argument).
\smallskip

\subsection{An improvement of Bakry-Emery criterion.}\label{subsecsuperyes}

As we remarked at this end of section \ref{secdrift} we may come back to the initial inequality in \eqref{eqdist1} which becomes in our new situation
\begin{equation}\label{eqdist3super}
\eta_t  \leq   \eta_0 \,  - \, K \, \int_0^t \, \varphi(\eta_s) \, ds + 2 \, \int_0^t \, \E^{\Q}\left(|z_s-\omega_s| \, |\nabla \log P_{T-s}h(\omega_s)|\right) ds  \, ,
\end{equation}
and yields 
\begin{equation}\label{eqdist4super}
\eta'_t \leq  - \, K \, \varphi(\eta_t) +  2 \, \eta^{\frac12}_t \, \left(\int \frac{|\nabla h|^2}{h} \, d\mu_{T}\right)^{\frac 12} \, .  
\end{equation}
(here again (H.C.$0$) is satisfied so that, for short, $|\nabla P_s| \leq P_s |\nabla|$.) To explore \eqref{eqdist4super} we shall use both the remark \ref{remcgm} and the usual trick $ab \leq \lambda a^2 + \frac 1\lambda \, b^2$ for $a,b,\lambda$ positive. Hence
\begin{equation}\label{eqdist5super}
\eta'_t \leq \, \left( -K \, \frac{\varphi(\varepsilon)}{\varepsilon} + 2\lambda \right)\, \eta_t + \left(\frac{2}{\lambda} \, \left(\int \frac{|\nabla h|^2}{h} \, d\mu_{T}\right) + K \varphi(\varepsilon) \right) \, . 
\end{equation}
We deduce, denoting $A=K \, \frac{\varphi(\varepsilon)}{\varepsilon} - 2\lambda$,  
$$
\eta_T \leq \eta_0 \, e^{-AT} + (1-e^{-AT}) \, \frac{ \frac 2\lambda \left(\int \frac{|\nabla h|^2}{h} \, d\mu_{T}\right) + K \varphi(\varepsilon)}{A} \, .
$$
Choose $\lambda= (1/4) \, K \, (\varphi(\varepsilon)/\varepsilon)$ so that $A=(1/2) \, K \, (\varphi(\varepsilon)/\varepsilon) > 0$. $\eta_T$ is thus bounded in time, but the bound is not tractable except for $T=+\infty$ (starting with $\mu_0=\mu$) or if $\eta_0=0$. In both cases we have obtained 
\begin{equation}\label{eqsuperenfin}
W_2^2(h\mu_T,\mu_T) \leq  \varepsilon \, + \, \left(\frac{8 \varepsilon^2}{K^2 \, \varphi^2(\varepsilon)}\right) \, \int \frac{|\nabla h|^2}{h} \, d\mu_T \, .
\end{equation}
It remains to optimize in $\varepsilon$. In full generality we choose $\varepsilon$ such that both terms in the sum of \eqref{eqsuperenfin} are equal (we know that we are loosing a factor less than $2$).  Remark that we do not use the explicit form of $\varphi$, i.e. we may replace (H.$\varphi$.K) by (H.$\alpha$.K) in what we did previously. We have thus obtained 
\begin{proposition}\label{propsuperweakw}
Assume that $U$ satisfies (H.$\alpha$.K) for $K>0$. Let $F$ be the inverse (reciprocal) function of $\varepsilon \mapsto \varepsilon \, \alpha^2(\varepsilon)$. Denote  $\mu_T= P(T,x,.)$ and $\mu_\infty(dx) =\mu(dx) = e^{-U(x)} \, dx$. Then for all $0<T\leq +\infty$, $\mu_T$ satisfies for all nice $h$, $$W_2^2(h\mu_T,\mu_T) \leq 2 \, F\left(\frac{8}{K^2} \, I_T(h)\right) \, ,$$ where $I_T(h) = \int \frac{|\nabla h|^2}{h} \, d\mu_T$ is the Fisher information of $h$.
\end{proposition}
When $F$ is equal to identity, such an inequality is called a $W_2I$ inequality (see \cite{GL} definition 10.4). Here we obtained a weak form of $W_2I$ inequality (which is clear on \eqref{eqsuperenfin}) in the spirit of the weak Poincar\'e or the weak log-Sobolev inequalities. 

In particular, using (H.W.I), since (H.C.0) is satisfied we obtain 
\begin{corollary}\label{corsuperweak}
Under the hypotheses of proposition \ref{propsuperweakw}, $\mu$ satisfies the inequality $$H(h\mu|\mu) \leq 2 \, \left(I(h) \, F\left(\frac{8}{K^2} \, I(h)\right)\right)^{\frac 12} \, .$$
\end{corollary}
Weak logarithmic Sobolev inequalities were introduced and studied in \cite{CGG}. Actually, we are not exactly here in the situation of \cite{CGG} because we wrote the previous inequality in terms of a density of probability.   Let $h=f^2/\int f^2 d\mu$. We deduce from the previous corollary
$$
\int \, f^2 \, \log\left( \frac{f^2}{\int f^2 d\mu}\right) \, d\mu \leq \, 4 \, \left(\int f^2 d\mu\right)^{\frac 12} \, \left(\int |\nabla f|^2 \, d\mu\right)^{\frac 12} \, F^{\frac 12}\left(\frac{32}{K^2} \, \frac{\int |\nabla f|^2 \, d\mu}{\int f^2 \, d\mu} \, \right)\, ,
$$
so that if $F(\lambda \, a) \leq \theta(\lambda) \, F(a)$, $$
\int \, f^2 \, \log\left( \frac{f^2}{\int f^2 d\mu}\right) \, d\mu \leq \, 4 \, \theta^{\frac 12}\left(\frac{32}{K^2}\right) \, G_1\left(\int f^2 d\mu\right) \, G_2\left(\int |\nabla f|^2 \, d\mu\right)\, ,$$
where $G_1(a)=a^{\frac 12} \, \theta^{\frac12}(1/a)$ and $G_2(a)=a^{\frac 12} \, F^{\frac12}(a)$. The previous inequality looks like the Nash inequality version of a weak log-Sobolev inequality, but with the $\L^2$ norm of $f$ in place of the $\L^\infty$ norm of $f - \int f d\mu$. So the previous inequality is not only ``weak'' but also ``defective''.
\medskip

\subsubsection{Super convex potentials.}\label{subsubsuper}

In this sub(sub)section we assume that  $\varphi(a)=a^\beta$ for some $\beta \geq 1$, so that $F(a)=a^{\frac{1}{2\beta-1}}$. We thus have
\begin{equation}\label{eqversls1}
\int \, f^2 \, \log\left(\frac{f^2}{\int f^2 d\mu}\right) \, d\mu \leq \, 4 \, \left(\frac{32}{K^2}\right)^{\frac{1}{2(2\beta-1)}} \, \left(\int f^2 d\mu\right)^{\frac{\beta-1}{2\beta- 1}} \, \left(\int |\nabla f|^2 \, d\mu\right)^{\frac{\beta}{2\beta -1}} \, .
\end{equation}
Recall first that if $g\geq 0$, then $\Var_\mu(g) \leq \Ent_\mu(g)$ (see e.g. \cite{CGJMP} (2.6)).\\ Next recall the following: defining $m_\mu(g)$ as a \emph{median} of $g$, we have
\begin{equation}\label{eqmed}
\Var_\mu(g) \, \leq \, 4 \, \int (g-m_\mu(g))^2 \, d\mu \, \leq \, 36 \, \Var_\mu(g) \, .
\end{equation}
We may decompose $f-m_\mu(f)= (f-m_\mu(f))_+ - (f-m_\mu(f))_-=g_+-g_-$ so that both $g_+$ and $g_-$ are non negative with median equal to $0$. In addition, if $f$ is Lipschitz, so are $g_+$ and $g_-$, $\nabla f=\nabla g_+ + \nabla g_-$, and the product of both vanishes. Hence 
$$
\Var_\mu(f) \leq 4 \left(\int (g_+)^2 d\mu + \int (g_-)^2 d\mu\right) \, ,$$  while
\begin{eqnarray*}
\int (g_+)^2 d\mu &\leq&  9 \, \Var_\mu(g_+) \, \leq \, 9 \, \Ent_\mu(g_+)\\ &\leq& 36  \, \left(\frac{32}{K^2}\right)^{\frac{1}{2(2\beta-1)}} \, \left(\int (g_+)^2 d\mu\right)^{\frac{\beta-1}{2\beta- 1}} \, \left(\int |\nabla g_+|^2 \, d\mu\right)^{\frac{\beta}{2\beta -1}} \, .
\end{eqnarray*}
It follows from \eqref{eqversls1} $$\int (g_+)^2 d\mu \leq (36)^{\frac{2\beta-1}{\beta}} \, \left(\frac{32}{K^2}\right)^\frac{1}{2\beta} \, \int |\nabla g_+|^2 \, d\mu \, ,$$ similarly for $g_-$. We have thus obtained 
\begin{theorem}\label{thmsuperconvex}
Assume that $U$ satisfies (H.$\varphi$.K) for $K>0$ and $\varphi(a)=a^\beta$, for $\beta \geq 1$.  Then, $\mu$ satisfies both a Poincar\'e inequality with $$C_P(\mu) \leq C_P(K,\beta) = 4 \, (36)^{\frac{2\beta-1}{\beta}} \, \left(\frac{32}{K^2}\right)^\frac{1}{2\beta} \, ,$$ and a log-Sobolev inequality with $$C_{LS}(\mu) \leq C_{LS}(K,\beta) =\left(\frac{32}{K^2}\right)^\frac{1}{2\beta}  \, \left(4^{\frac{3\beta-2}{2\beta-1}} 36^{\frac{\beta-1}{\beta}} + 8 \times 36^{\frac{2\beta-1}{\beta}}\right) \, . $$
\end{theorem}
\begin{proof}
The statement on the Poincar\'e inequality follows from the previous discussion. \\ Concerning the log-Sobolev inequality, let $\tilde f= f - \int f d\mu$. Then, Rothaus lemma (see \cite{Ane} lemma 4.3.7) says that $$ \Ent_\mu(f) \leq \Ent_\mu( \tilde f) + 2 \, \Var_\mu(f) \, .$$ Applying \eqref{eqversls1} to $\tilde f$ together with the Poincar\'e inequality, yield the result (after some elementary calculation).
\end{proof}
\smallskip

\begin{remark}\label{rembobsphere} \quad
This theorem applies in particular to $U(x)=|x|^{2\beta}$ for $\beta \geq 1$, according to proposition \ref{propsuper}. The fact that $\mu$ satisfies a log-Sobolev inequality in this situation is well known, but here we obtain an explicit (though not really cute) expression for the constant that only depends on $\beta$ and not on the dimension $n$. \\ Unfortunately, in this particular situation, our bounds are not optimal. Indeed, spherically symmetric log-concave probability measures are now well understood.

For the Poincar\'e constant, it was shown by Bobkov \cite{bobsphere} that  $$\frac{1}{n} \, \Var_\mu(x) \, \leq \, C_P(\mu) \leq \frac{13}{n} \, \Var_\mu(x)  \, .$$ It is an (easy) exercise to see that $\Var_\mu(x) = \Gamma((n+2)/2\beta)/\Gamma(n/2\beta)$, so that $C_P(\mu) \leq c(\beta) \, n^{\frac 1\beta -1}$ which goes to $0$ as $n \to +\infty$. \\ A famous conjecture by Kannan-Lovasz-Simonovitz is that the previous bound for spherically symmetric measures extends (up to a change of the constant $13$) to any log-concave measure. If true, the KLS conjecture will presumably give a better upper bound for the Poincar\'e constant than ours.

Regarding the log-Sobolev constant, the work by Huet \cite{Nolwen}, furnishes a lower bound for the isoperimetric profile of $\mu$ (see Theorem 3 and the discussion p.98 therein) which indicates a similar bound for the log-Sobolev constant as above, i.e. depending on the isotropic constant ($n^{\frac{1-\beta}{2 \beta}}$) of $\mu$. \hfill $\diamondsuit$
\end{remark}
\medskip

\subsubsection{Lack of uniform convexity.}\label{subsublack}

Now choose $\alpha(a)=a^\beta$ for some $\beta \geq 1$ and $a\leq 1$, and $\alpha(a)= 1$ for $a\geq 1$. (H.$\alpha$.K) is less restrictive than before since it only implies a linear behavior at infinity for the gradient of the potential.

 We now have $F(a)=a^{\frac{1}{2\beta+1}}$ for $a\leq 1$ and $F(a)=a$ for $a \geq 1$. It follows 
\begin{equation}\label{eqversls2}
\Ent_\mu(f) \leq \, 4 \, \left(\frac{32}{K^2}\right)^{\frac{1}{2(2\beta+1)}} \, \left(\int f^2 d\mu\right)^{\frac{\beta}{2\beta+1}} \, \left(\int |\nabla f|^2 \, d\mu\right)^{\frac{\beta+1}{2\beta +1}} 
\end{equation}
if $
\int |\nabla f|^2 \, d\mu \leq \frac{K^2}{32} \, \int f^2 d\mu$ and
\begin{equation}\label{eqversls3}
\Ent_\mu(f) \leq \, \frac{32}{K} \,  \int |\nabla f|^2 \, d\mu \quad \textrm{ otherwise.}
\end{equation}
Proceeding as before we obtain
\begin{theorem}\label{thmlack}
Assume that $U$ satisfies (H.$\alpha$.K) for $K>0$ and $\alpha(a)=a^\beta\wedge 1$ for $\beta \geq 1$.  Then,  $\mu$ satisfies both a Poincar\'e inequality with $$C_P(\mu) \leq C_P(K,\beta) = \max \left(\frac{32}{K} \, , \, 4 \, (36)^{\frac{2\beta+1}{\beta+1}} \, \left(\frac{32}{K^2}\right)^\frac{1}{2(\beta+1)} \right) \, ,$$ and a log-Sobolev inequality with $$C_{LS}(\mu) \leq C_{LS}(K,\beta) = \max \left(\frac{32}{K} \, , \, \left(\frac{32}{K^2}\right)^\frac{1}{2(\beta+1)}  \, \left(4^{\frac{3\beta+1}{2\beta+1}} 36^{\frac{\beta}{\beta+1}} + 8 \times 36^{\frac{2\beta+1}{\beta+1}}\right)\right) \, . $$
\end{theorem}
\medskip

\begin{remark}\label{remweakw}
Using the general form of the (H.W.I) inequality it is quite easy to adapt the previous proof in order to show the following result :\\ \quad \textit{Let $\mu$ satisfying (H.C.K) for some $K>-\infty$. If $\mu$ satisfies a weak $W_2I$ inequality, $$W_2^2(h\mu,\mu) \leq C \, (I(h))^\beta$$ for some $0<\beta \leq 1$, then $\mu$ satisfies a log-Sobolev inequality with a constant depending on $C,K,\beta$ only. In particular $\mu$ satisfies a $T_2$ inequality.} 

In particular if we know that $\mu_T$ has a bounded below curvature, the previous theorems extend to $\mu_T$. \hfill $\diamondsuit$
\end{remark}
\medskip


\section{\bf Using reflection coupling.}\label{secreflect}

As we have seen in the previous section, the simple coupling using the same Brownian motion is not fully well suited to deal with non uniformly convex potentials.\\ In a recent note \cite{Ebe}, Eberle studied the contractivity property in Wasserstein distance $W_1$, induced by another well known coupling method: coupling by reflection (or mirror coupling) introduced in \cite{LR}. We shall see now how to use this coupling method in the spirit of what we have done before.

\subsection{Reflection coupling for the drifted brownian motion.}\label{subsecreflec}
In this section we consider $X_t^x$ the solution starting from $x$ of the Ito stochastic differential equation
\begin{equation}\label{eqitoref}
dX_t =  dB_t + b(X_t) \, dt \, ,
\end{equation}
where $b$ is smooth enough. We introduce another formulation of the semi-convexity property, namely :
\begin{equation}\label{defkappa}
\kappa(r) = \inf \, \left\{ - \, 2 \, \frac{\langle b(x)-b(y),x-y\rangle}{|x-y|^2} \, ; \, |x-y|=r \right\} \, ,
\end{equation}
so that, it always holds $$2 \, \langle b(x)-b(y),x-y\rangle \, \leq \, -\kappa(|x-y|) \, |x-y|^2 \, .$$ We shall say that {\bf(H.$\kappa$)} is satisfied if 
\begin{equation}\label{eqkappa}
\liminf_{r \to + \infty} \, \kappa(r) = \kappa_\infty > 0 \, .
\end{equation}
This condition is typically some ``uniform convexity at infinity'' condition. Indeed if $b=- \frac 12 \, \nabla U$ where $U=U_1+U_2$ with $U_1$ satisfying (H.C.$\kappa_\infty$) and $U_2$ compactly supported, then (H.$\kappa$) is satisfied. We shall come back later to this. Notice that if \eqref{eqkappa} is satisfied, the solution of \eqref{eqitoref} is strongly unique and non explosive, using the same tools as we used before.\\ Now, following \cite{Ebe} we introduce (with some slight change of notations)
\begin{eqnarray}\label{eqdefkappa}
R_0 &=& \inf \left\{R\geq 0 \, ; \, \kappa(r)\geq 0 \, , \, \forall r\geq R\right\} \, , \\ R_1 &=& \inf \left\{R\geq R_0 \, ; \, \kappa(r)\geq 8/(R(R-R_0)) \, , \, \forall r\geq R\right\} \, , \nonumber\\ \varphi(r) &=& \exp\left(- \, \frac 14 \, \int_0^r \, s \kappa^-(s) \, ds\right), \quad \Phi(r) = \int_0^r \, \varphi(s) \, ds \, , \nonumber \\ g(r) &=& 1 - \frac 12 \left(\int_0^{r\wedge R_1} \frac{\Phi(s)}{\varphi(s)} \, ds \, \Big/ \, \int_0^{R_1} \frac{\Phi(s)}{\varphi(s)} \, ds \right) \nonumber \\ D(r) &=& \int_0^r \, \varphi(s) \, g(s) \, ds \, . \nonumber
\end{eqnarray}
Notice that $$\frac 12 \leq g \leq 1 \, \textrm{ and } \, \exp\left(- \, \frac 14 \, \int_0^{R_0} \, s \kappa^-(s) \, ds\right)=\varphi_{min} \leq \varphi \leq 1 \, .$$ If (H.$\kappa$) is satisfied, $R_0<+\infty$ so that $\varphi_{min}>0$ and $$\frac{\varphi_{min}}{2} \, r \, \leq D(r) \, \leq \, r \, ,$$ i.e. $D(|x-y|)$ which is actually a distance, is equivalent to the euclidean distance.\\ Hence a consequence of Theorem 1 in \cite{Ebe} is the following :
\begin{theorem}\label{thmeberle}
Assume that (H.$\kappa$) is satisfied. Let $\lambda$ be defined by $$\frac 1\lambda = \int_0^{R_1} \frac{\Phi(s)}{\varphi(s)} \, ds \, \leq \, \frac{R_1^2}{\varphi_{min}} \, .$$ Then for all initial distributions $\nu$ and $\mu$, and all $t$, the $W_1$ Wasserstein distance satisfies $$W_1(\nu_t,\mu_t) \leq \frac{2}{\varphi_{min}} \, e^{-\lambda t} \, W_1(\nu,\mu) \, .$$
\end{theorem}
In order to prove this result, Eberle adds to \eqref{eqitoref} the following Ito s.d.e.
\begin{equation}\label{eqitoref2}
dY_t =  (Id -2 e_t e_t^*) \, dB_t + b(Y_t) \, dt \, , \, Y_0=y \, ,
\end{equation}
where $e_t=(X_t-Y_t)/|X_t-Y_t|$ and $e^*$ is the transposed of $e$ (remark that if $n=1$, it just changes $B$ into $-B$). Of course one has to consider \eqref{eqitoref} and \eqref{eqitoref2} together. Existence, strong uniqueness and non explosion are again easy to show. Now introduce the coupling time $T_c$ defined by $$T_c = \inf \, {t\geq 0 \, ; \, X_t=Y_t} \, ,$$ and finally define $$\bar X_t^y=Y_t^y \, \textrm{ if } t\leq T_c \, , \quad \bar X_t^y=X_t^x \, \textrm{ if } t\geq T_c \, .$$ It is easy to see that $X_.^y$ and $\bar X_.^y$ have the same law, so that the distribution of $(X_t^x,\bar X_t^y)$ is a coupling of $P(t,x,.)$ and $P(t,y,.)$. Of course this extends to any initial distributions $(\mu,\nu)$ and furnishes a coupling of $(\mu_t,\nu_t)$.

It follows that $Z_t=X_t^x - \bar X_t^y$ solves
\begin{equation}\label{eqdiffeber}
dZ_t = (b(X_t^x)-b(\bar X_t^y))dt + 2 \, \frac{Z_t}{|Z_t|} \, dW_t
\end{equation}
where $W_t= \int_0^t \, e_s^* \, dB_s$ is a one dimensional brownian motion.

The key of the proof of Theorem \ref{thmeberle} is then that, if $r_t=D(|X_t^x - \bar X_t^y|)$, $r_.$ is a semi-martingale with decomposition
\begin{equation}\label{eqdecomp}
r_t = D(|x-y|) + \int_0^{t\wedge T_c} \, 2 \, \varphi(r_s) \, g(r_s) \, dW_s + \int_0^{t\wedge T_c} \, \beta_s \, ds \, ,
\end{equation}
where the drift term satisfies 
\begin{equation}\label{eqdriftcontrol}
\beta_s \leq - \, \lambda \, r_s \, .
\end{equation}
Taking expectation, it immediately shows that the $W_D$ Wasserstein distance decays exponentially fast. As remarked by several authors, one can then deduce as we did previously
\begin{equation}\label{eqdecayunif}
|\nabla P_t f| \leq \frac{2}{\varphi_{min}} \, e^{-\lambda t} \, \parallel \nabla f\parallel_\infty \, .
\end{equation} 
\smallskip

As a consequence we obtain
\begin{theorem}\label{thmconvinfty}
Assume that $b=-\frac 12 \, \nabla U$ satisfies(H.$\kappa$) and that $\mu=(1/Z_U) e^{-U}$ is a probability measure. Then $\mu$ satisfies a Poincar\'e inequality with constant $C_P \leq (1/2\lambda)$.
\end{theorem}
\begin{proof}
Recall that $$ \Var_\mu(P_t f) = \frac 12 \, \int_t^{+\infty} \, \int \, |\nabla P_sf|^2 \, d\mu \, ds \, .$$ According to \eqref{eqdecayunif}, we thus have $$\Var_\mu(P_t f) \leq \frac{1}{\lambda \, \varphi^2_{min}} \, e^{- \, 2 \, \lambda t} \, \parallel \nabla f\parallel^2_\infty $$ for all Lipschitz function $f$. According to Lemma 2.12 in \cite{CGZ}, we deduce that $\Var_\mu(P_t f) \leq e^{- \, 2 \, \lambda t} \, \Var_\mu(f)$, hence the result.
\end{proof}
\begin{remark}
One can see that the reflection coupling cannot furnish some information on $W_2$, the concavity of $D$ near the origin being crucial. In the same negative direction, theorem \ref{thmconvinfty} cannot be extended to the log-Sobolev framework, the key lemma 2.12 in \cite{CGZ} being restricted to the variance control. \hfill $\diamondsuit$
\end{remark}

\begin{example}\label{subsecexamp}
\begin{enumerate} \item[(1)] \quad If (H.$\varphi$.K) is satisfied with $\varphi(a)=a^\beta$, we have $\kappa(a)=a^{2(\beta-1)}$. We thus have $R_0=0$, $R_1= (8/K)^{\frac{1}{2\beta}}$, $\varphi_{min}=1$ and finally $$C_P \leq \frac 12 \, \left(\frac 8K\right)^{\frac 1\beta} \, .$$ We recover the result in Theorem \ref{thmsuperconvex}, i.e. a bound $C_\beta \, K^{- \, 1/\beta}$ but with a better constant $C_\beta$. \item[(2)] \quad If (H.$\alpha$.K) is satisfied, one can choose $R_0=\sqrt{2 \varepsilon}$, $R_1 = \sqrt{2 \varepsilon} + (4/\sqrt{ K \alpha(\varepsilon)})$, $\varphi_{min}=\exp \left( - \, \frac 14 \, \varepsilon^2 \, K \alpha(\varepsilon)\right)$ and finally $$ C_P \leq \left(2\varepsilon + \frac{16}{K \alpha(\varepsilon)}\right) \, e^{K \varepsilon^2 \, \alpha(\varepsilon)/4}  \, .$$ \item[(3)] \quad Now assume that the potential $U$ can be written $U=U_1+U_2$ where $U_1$ satisfies (H.C.K) for some $K>0$ and $U_2$ satisfies $\parallel \nabla U_2\parallel_\infty = M < +\infty$. It easily follows that (H.$\kappa$) is satisfied with $\kappa(a)=K - \frac Ma$. We thus have $$R_0=\frac MK \, ,  \quad R_1=\frac MK + \sqrt{\frac 8K} \, , \quad \varphi_{min} = e^{- \, \frac{M^2}{8K}} \, .$$ We finally obtain $$C_P \leq  \left(\frac{\parallel \nabla U_2\parallel_\infty}{K} + \sqrt{\frac 8K}\right)^2 \, \exp \left(\frac{\parallel \nabla U_2\parallel_\infty^2}{8K}\right)  \, .$$ An old result by Miclo (unpublished but explained in \cite{Ledgibbs}) indicates that such a result (without the square of the supremum of the gradient but without $K$ in the exponential) can be obtained by using the usual Holley-Stroock perturbation argument. \hfill $\diamondsuit$
\end{enumerate}
\end{example}

\subsection{The log-concave situation.}
Now consider the situation where $b$ satisfies (H.C.0). In this situation we have $\lambda=0$ ($R_1=+\infty$, $\varphi=g=1$) so that \eqref{eqdecomp} becomes $$dr_t =  2 \, dW_t + \beta_t \, dt$$ for $t\leq T_c$ with $\beta_t \leq 0$. It follows that $r_t \leq |x-y| + 2 \, W_t$ up to the first time $T_{|x-y|}$ the brownian motion $W_.$ hits $- \, |x-y|/2$.\\ In particular, $$\P(r_t>0) \, \leq \, \P(t<T_{|x-y|}) \, \leq \, \frac{|x-y|}{\sqrt {2\pi \, t}} \, ,$$ since the law of $T_{|x-y|}$ is given by $$\P(T_{|x-y|} \in da) = \frac{|x-y|}{2 \sqrt {2\pi \, a^3}} \, e^{-|x-y|^2/8a} \, \BBone_{a>0} \, da \, .$$ As a first by-product we obtain
\begin{proposition}\label{proppoinreverse}
If $b$ satisfies (H.C.0) then $$|\nabla P_t f| \leq \frac{2}{\sqrt {2\pi \, t}} \,  \, \parallel f\parallel_\infty \, .$$
\end{proposition}
Actually if $b= - \, \frac 12 \, \nabla U$ with $U$ convex (i.e. in the zero curvature situation of the $\Gamma_2$ theory) the inequality $|\nabla P_t f| \leq \frac{1}{\sqrt {t}} \,  \, \parallel f\parallel_\infty$ is well known as a consequence of what is called the reverse (local) Poincar\'e inequality (see \cite{Ane}). The previous proposition extends this result (up to the constant) to a non-gradient drift.
\begin{proof}
Recall that $X_t^x=\bar X_t^y$ for $t>T_c$. It follows
\begin{eqnarray*}
P_tf(x)-P_tf(y) &=& \E\left((f(X_t^x)-f(\bar X_t^y)) \, \BBone_{T_c>t}\right) \\ &\leq& 2 \, \parallel f\parallel_\infty \, \P(T_c>t) \, \leq \, 2 \, \parallel f\parallel_\infty \, \P(t<T_{|x-y|}) \\  &\leq& \, \frac{2\, \parallel f\parallel_\infty}{\sqrt {2\pi \, t}} \, |x-y| \, ,
\end{eqnarray*}
hence the result.
\end{proof}
If one wants to get a contraction bound for the gradient (in the spirit of \eqref{eqdecayunif} or better of proposition \ref{propgrad}) we cannot only use a comparison with the brownian motion for which $\nabla P_t f = P_t \nabla f$.

\begin{remark}\label{reminterpol}
In the symmetric situation ($b= -\frac 12 \, \nabla U$) it is known that $t \mapsto \int \, |\nabla P_tf|^2 \, d\mu$ is non increasing. It easily follows that $$\parallel \nabla P_tf\parallel_{\L^2(\mu)} \leq \frac{\sqrt 2}{\sqrt t} \, \parallel  f\parallel_{\L^2(\mu)} \, .$$ If (H.C.$0$) is satisfied, this remark together with Proposition \ref{proppoinreverse} and Riesz-Thorin interpolation theorem show that, up to an universal constant, the same holds in all $\L^p(\mu)$ spaces. \hfill $\diamondsuit$
\end{remark}

\begin{remark}\label{remebtata}
If we assume that (H.$\kappa$) is satisfied, we may replace the comparison with a Brownian motion by the comparison with an Ornstein-Uhlenbeck process with parameter $\lambda/2$, according to standard comparison theorems for one dimensional Ito processes (see e.g \cite{IW} Chapter VI theorem 1.1). For the O-U process, it is known (see \cite{PY}) that $$\P(T_{|x-y|} \in da) = \frac{|x-y|}{2 \sqrt {2\pi}} \, \left(\frac{\lambda}{2 \sinh(a\lambda/2)}\right)^{\frac 32} \, e^{\left(- \, \frac{\lambda \, |x-y|^2 e^{-a\lambda/2}}{16 \sinh(a\lambda/2)} \, + \, \frac{a\lambda}{4} \right)} \, da \, .$$ An explicit bound for $\P(t<T_{|x-y|})$ can be obtained by using the reflection principle in \cite{Yi}, namely $$\P(t<T_{|x-y|}) \leq \frac{\sqrt \lambda \, e^{-t \lambda/2}}{\sqrt{2 \pi} \, \sqrt{1 - e^{-t\lambda}}} \, |x-y| \, ,$$ yielding $$|\nabla P_t f| \leq  \,  \frac{2}{\varphi_{min}} \, \frac{\sqrt \lambda \, e^{-t \lambda/2}}{\sqrt{2 \pi} \, \sqrt{1 - e^{-t\lambda}}} \, \parallel f\parallel_\infty \, .$$ 
These bounds are interesting as regularization bounds (from bounded to Lipschitz functions), but notice that we have lost a factor $2$ in the exponential decay.
\hfill $\diamondsuit$
\end{remark}
\medskip

\subsection{Reflection coupling for general diffusions.}\label{subsecrefcoupsigma}

The case of  a general diffusion process with a non constant diffusion matrix as in section \ref{secdiff} is more delicate to handle, as already remarked in \cite{LR} (Theorem 1). 

Assume that $\sigma$ is a bounded and smooth square matrices field and that it is uniformly elliptic. The quantities we need here are (notations differ from \cite{LR})
\begin{equation}\label{eqbornes}
M = \sup_x \, \sup_{|u|=1} \, |\sigma(x) \, u|^2 \, , \, N = \sup_x \, \sup_{|u|=1} \, |\sigma^{-1}(x) \, u|^2 \, , \, \Lambda = \sup_{x,x'} \, \sup_{|u|=1} \, |(\sigma(x)-\sigma(x')) \, u|^2 \, .
\end{equation}
Recall the Lindvall-Rogers reflection coupling
\begin{eqnarray*}
dX_t &=& \sigma(X_t)dB_t + b(X_t) dt \, ,\\
dX'_t &=& \sigma(X'_t) \, H_t \, dB_t + b(X'_t) dt \, ,
\end{eqnarray*}
where $$H_t \, = \, Id \, - \, 2 \, \left(\frac{\sigma^{-1}(X'_t) \, (X_t-X'_t)}{|\sigma^{-1}(X'_t) (X_t-X'_t)|}\right)\left(\frac{\sigma^{-1}(X'_t) \, (X_t-X'_t)}{|\sigma^{-1}(X'_t) (X_t-X'_t)|}\right)^* \, .$$ Existence and strong uniqueness can be shown as previously. Of course, as in subsection \ref{subsecreflec}, we replace $X'_t$ by $X_t$ if $t>T_c$ the coupling time, but not to introduce new notation we still use $X'_.$. 

As in \cite{LR} define $$Y_t=X_t-X'_t \, , \, V_t=\frac{Y_t}{|Y_t|} \, , \, \alpha_t=\sigma(X_t) - \sigma(X'_t)H_t \, , \, \beta_t =b(X_t)-b(X'_t) \, .$$ According to (15) in \cite{LR} we have $$d(|Y_t|) = \langle V_t, \alpha_t \, dB_t\rangle \, + \, \frac 12 \, \frac{1}{|Y_t|} \, \left(2\langle Y_t,\beta_t\rangle + Trace(\alpha_t \, \alpha^*_t) - |\alpha^*_t \, V_t|^2\right) \, ,$$ and a simple calculation shows that $$Trace(\alpha_t \, \alpha^*_t) - |\alpha^*_t \, V_t|^2 = $$ $$ = \, Trace((\sigma(X_t)-\sigma(X'_t)) \,(\sigma(X_t)-\sigma(X'_t))^*) \, - \, |(\sigma(X_t)-\sigma(X'_t))^* \, V_t|^2 \, ,$$ while 
$$|\alpha^*_t \, V_t|^2 \geq \frac 2N - \Lambda \, .$$

Applying Ito formula we thus have for a smooth function $D$
\begin{eqnarray*}
\E(D(|Y_t|)) &=& \frac 12 \, \E\left(\frac{D'(|Y_t|)}{|Y_t|} \left(2\langle Y_t,\beta_t\rangle + Trace(\alpha_t \, \alpha^*_t) - |\alpha^*_t \, V_t|^2\right)+ \, D''(|Y_t|)|\alpha^*_t V_t|^2\right)
\end{eqnarray*}
We introduce the natural generalization of (H.$\kappa$), namely we assume that
\begin{equation}\label{eqkappagene}
\textrm{for a $\kappa$ satisfying \eqref{eqkappa}, }|\sigma(x)-\sigma(y)|_{HS}^2 + 2 \, \langle b(x)-b(y),x-y\rangle \, \leq  \, - \kappa(|x-y|) \, |x-y|^2  \,   .
\end{equation}
If $D$ is a non decreasing, concave function we thus get, provided $(2/N)-\Lambda>0$, $$2 \, \E(D(|Y_t|)) \leq  \E \left(- \, D'(|Y_t|) \, \kappa(|Y_t|) |Y_t| + D''(|Y_t|) \left(\frac 2N - \Lambda\right)\right) \, .$$ Hence looking carefully at the calculations in \cite{Ebe}, we see that, provided $(2/N)-\Lambda>0$, the only thing we have to change in \eqref{eqdefkappa} is the definition of $\varphi$ replacing $1/4$ by the inverse of $(2/N)-\Lambda>0$, all other definitions being unchanged. We have thus obtained
\begin{theorem}\label{thmeberetendu}
Assume that \eqref{eqbornes} and \eqref{eqkappagene} are satisfied. Assume in addition that $(2/N)-\Lambda>0$. Then defining $$\varphi_{min}=e^{- \, \frac{1}{(2/N)-\Lambda} \, \int_0^{R_0} s\kappa^-(s) \, ds} \, ,$$ the conclusion of Theorem \ref{thmeberle} is still true with $\lambda = \frac 12 \, (\varphi_{min}/R_1^2)$.
\end{theorem}
All the consequences of Theorem \ref{thmeberle} still hold (up to the modifications of the constants), in particular one can extend (H.C.K) to the situation of ``convexity at infinity'' as in Example \ref{subsecexamp} (3). Details are left to the reader.
\medskip

As we already said, the condition $(2/N)-\Lambda>0$ already appears in \cite{LR} and ensures that the coupling by reflection is succesfull. Roughly speaking it means that the fluctuations of $\sigma$ are not too big with respect to the uniform ellipticity bound. 

\subsection{Gradient commutation property and reflection coupling}

It is of course quite disappointing at first glance that the only gradient commutation property that we get using this nice contraction results in $W_1$ distance, is restricted to Lipschitz function as in \eqref{eqdecayunif}. Let us see however that we may transfer this to stronger gradient commutation properties in some cases. The main tool is the following lemma on H\"older's type inequality in Wasserstein distance.

\begin{lemma}
Suppose that $\nu$ and $\mu$ are two probability measures on $\mathbb{R}$, then
 for all $q>1$ and $p$ such that $\frac1p+\frac1q=1$, we have
 \begin{equation}\label{Wholder}
 W_2(\nu,\mu)\le W_1^{\frac1{2q}}(\nu,\mu)\,W_{(2-\frac1q)p}^{1-\frac1{2q}}(\nu,\mu).\end{equation}
 Furthermore the result tensorises in the sense, that if for $i=1,...,n$, $\mu_i$ and $\nu_i$ are probability measures on $\mathbb{R}$, we have for some constant $c(n)$
 $$W_2(\otimes_1^n\nu_i,\otimes_1^n\mu_i)\le c(n)\,  W_1^{\frac1{2q}}(\otimes_1^n\nu_i,\otimes_1^n\mu_i)\,W_{(2-\frac1q)p}^{1-\frac1{2q}}(\otimes_1^n\nu_i,\otimes_1^n\mu_i).$$
\end{lemma}
\begin{proof}
The proof is indeed quite simple and relies mainly on H\"older's inequality. Indeed, in dimension one the optimal transport plan is the same for every convex cost (see for example Villani \cite{vil2}), so that there exists a transport plan $\pi$ such that
\begin{eqnarray*}
W_2^2(\nu,\mu)&=&\int\int |x-y|^2d\pi\\\
&\le&\left(\int\int|x-y|d\pi\right)^{1/q}\,\left(\int\int|x-y|^{(2-\frac1q)p}d\pi\right)^{1/p}\\
&=&W_1^{\frac1{q}}(\nu,\mu)\,W_{(2-\frac1q)p}^{2-\frac1q}(\nu,\mu).
\end{eqnarray*}
The case of product probability measure is deduced using the result in dimension one and the following two direct assertions
$$W_2^2(\otimes_1^n\nu_i,\otimes_1^n\mu_i)=\sum_{i=1}^nW_2^2(\nu_i,\mu_i),$$
and if $\nu$ and $\mu$ have for $ith$ marginal $\nu_i$ and $\mu_i$
$$W_p(\nu_i,\mu_i)\le W_p(\nu,\mu).$$
\end{proof}

\begin{remark}
We failed at the present time to get the general version of this lemma, i.e. does there exists a constant $c$ only depending on the dimension such that for two probability measures on $\mathbb{R}^n$, we have
$$W_2(\nu,\mu)\le c(n)\,W_1^{\frac1{2q}}(\nu,\mu)\,W_{(2-\frac1q)p}^{1-\frac1{2q}}(\nu,\mu)?$$
In fact, as will be seen from our applications, even if $c$ does depend of $\nu$ and $\mu$ (in a nice way), it would be sufficient to get new gradient commutation property. \hfill $\diamondsuit$
\end{remark}

We are now in position to prove various gradient commutation properties in non standard cases. For simplicity, we suppose here that the diffusion coefficient is constant, i.e.
$$dX_t=dB_t+b(X_t)dt.$$

\begin{theorem}\label{supercommut}
Let us suppose here that either $(X_t)$ lives in $\mathbb{R}$ or that starting from $X_0=x\in\mathbb{R}^n$, $X_t=(X^1_t,...,X^n_t)$ is composed of independent component. Assume moreover that $(H.\kappa)$ is satisfied  and that $\kappa(r)\ge-L$ then, with $\lambda$ defined in Theorem \ref{thmeberle},
$$W_2({\mathcal L}(X_t^x),{\mathcal L}(X_t^y))\le c(n) \left(\frac{2}{\phi_{min}}\right)^{\frac1{2q}}\, e^{\left[(1-\frac1{2q})L-\frac{\lambda}{2q}\right] t}\, |x-y|$$
so that the weak gradient commutation property holds
$$|\nabla P_t f|^2\le c(n) \left(\frac{2}{\phi_{min}}\right)^{\frac1{2q}}\, e^{\left[(1-\frac1{2q})L-\frac{\lambda}{2q}\right] t}\, P_t|\nabla f|^2$$
and thus a local Poincar\'e inequality holds.
\end{theorem}

Note that this theorem is the first one to give the commutation gradient property in non strictly convex case with a good behaviour at infinity.

\begin{proof}
Using synchronous coupling as previously explained and the fact that $\kappa(r)\ge -L$ we have that
$$W_{(2-\frac1q)p}({\mathcal L}(X_t^x),{\mathcal L}(X_t^y))\le e^{Lt}|x-y|.$$
In the same time, by Theorem  \ref{thmeberle}, we have that 
$$W_1({\mathcal L}(X_t^x),{\mathcal L}(X_t^y))\le\frac{2}{\phi_{min}} e^{-\lambda t}\,|x-y|.$$
We then use Lemma \ref{Wholder} to get the first assertion. The second one is obtained as in Proposition \ref{propgrad2}.
\end{proof}

\begin{example}
Consider for example the log-concave case $b(x)=-4x^3$ for which Bakry-Emery theory enables us to get that we are in 0-curvature and thus
$$|\nabla P_t f|^2\le P_t|\nabla f|^2.$$
However, using Theorem \ref{supercommut} and this last inequality, we easily get that there exists $\lambda>0$ such that
$$|\nabla P_t f|^2\le\min\left(1,\frac{2}{\phi_{min}}e^{-\lambda t}\right) P_t|\nabla f|^2,$$
which is completely new. It captures both the short time behavior equivalent to the $\Gamma_2$ 0-curvature criterion and the long time behavior for which $P_t f\to\mu(f)$ and thus $\nabla P_t f|$ is expected to decay to 0.\\
Note that we may extend this example to a double well potential, in the case when the height of the well is not too large.
\end{example}


\bigskip

\section{\bf Preserving curvature.}\label{secpreserv}

A natural question about curvature is the following: is curvature preserved by a diffusion process ?  According to a result by Kolesnikov \cite{Koles}, the Ornstein-Uhlenbeck process is essentially the only one, among diffusion processes, preserving log-concavity (i.e. if $\nu_0$ is log-concave, so is $\nu_T$ for all $T>0$). One may also wonder if $\nu_t$ may satisfy other ``curvature'' like inequality as $HWI$ for example. It would have important applications on local inequalities, indeed transportation inequalities together with a HWI inequality may imply logarithmic Sobolev inequality.

In the spirit of the previous remark, consider a standard Ornstein-Uhlenbeck process $X_.$, i.e. the solution of
\begin{equation}\label{eqOU}
dX_t = dB_t - \frac \lambda 2 \, X_t \, dt \, .
\end{equation}
The curvature $K$ is thus equal to $\lambda \in \R$. If $\mathcal L(X_0)=\nu$, it is known that the law $\nu_T$ of $X_T$ is the same as the law of $$e^{-\lambda T/2}\left(Z + \sqrt{\frac{e^{\lambda T}-1}{\lambda}} \, G\right) \, ,$$ where $G$ and $Z$ are independent random variables, $G$ being a standard gaussian variable and $Z$ having distribution $\nu$. Hence
\begin{equation*}
C_{LS}(\nu_T) \leq e^{-\lambda T} C_{LS}(\nu) + \frac{2(1-e^{-\lambda T})}{\lambda} \, ,
\end{equation*}
or, if we use the notation $C_{LS}(Y)=C_{LS}(\eta)$ for a random variable $Y$ with distribution $\eta$,
\begin{equation}\label{eqPOU}
C_{LS}\left(e^{-\lambda T/2}\left(Z + \sqrt{\frac{e^{\lambda T}-1}{\lambda}} \, G\right)\right) \leq e^{- \lambda T} C_{LS}(Z) +  \frac{2(1-e^{-\lambda T})}{\lambda} \, .
\end{equation}
But if $A$ is a random variable it is clear that $C_{LS}(\lambda A)=\lambda^2 \, C_{LS}(A)$. It follows 
\begin{equation}\label{eqPOU1}
C_{LS}\left(Z + \sqrt{\frac{e^{\lambda T}-1}{\lambda}} \, G\right) \leq  C_{LS}(Z) +  \frac{2(e^{\lambda T} - 1)}{\lambda} \, .
\end{equation}
The change of variable $Z=Z' + \sqrt{\frac{e^{\lambda T}-1}{\lambda}} \, G$, yields, using the symmetry of $G$ 
\begin{equation}\label{eqPOU2}
C_{LS}(Z) \leq  C_{LS}\left(Z + \sqrt{\frac{e^{\lambda T}-1}{\lambda}} \, G\right) +  \frac{2(e^{\lambda T} - 1)}{\lambda} \, .
\end{equation}
In particular, for $\lambda= - (1/\alpha^2)<0$ ($\alpha>0$), we may let $T$ go to $+ \infty$ and obtain
\begin{equation}\label{eqPOU3}
C_{LS}(Z) \leq  C_{LS}\left(Z + \alpha \, G\right) +  2 \, \alpha^2 \leq C_{LS}(Z) + 4 \, \alpha^2 \, .
\end{equation}
Similarly we have
\begin{equation}\label{eqPOU4}
C_{P}(Z) \leq  C_{P}\left(Z + \alpha \, G\right) +  \alpha^2 \leq C_{P}(Z) + 2 \, \alpha^2 \, .
\end{equation}

In particular, the distribution of $Z$ satisfies a log-Sobolev inequality if and only if, for some $\alpha>0$, the distribution of $Z+\alpha G$ satisfies a log-Sobolev inequality , and then the considered inequality is satisfied for all $\alpha$. Recall that $Z$ and $G$ are  \emph{independent}. This is not surprising since a more general result can be obtained directly (extending  a similar result for the Poincar\'e inequality in \cite{BBN} ):
\begin{proposition}\label{propBBN}
Let X and Y be independent random variables and $\lambda \in [0,1]$ then, $$C_{LS}(\sqrt \lambda X + \sqrt{1 - \lambda} Y) \leq \lambda C_{LS}(X) + (1-\lambda) C_{LS}(Y) \, , \textrm{ the same holds with $C_P$} \, .$$ Conversely, if Y is symmetric (i.e. Y and -Y have the same distribution), we also have $$ 
\lambda \, C_{LS}(X) \leq C_{LS}(\sqrt \lambda X + \sqrt{1 - \lambda} Y) + (1-\lambda) C_{LS}(Y) \, , \textrm{ the same holds with $C_P$} \, .$$ 
\end{proposition}

\begin{proof}
The first result for $C_P$ is proved in \cite{BBN} proposition 1. For $C_{LS}$ the proof is very similar. Let $f$ be smooth. Then $$ \E \, ((f^2 \, \log f^2)(\sqrt \lambda X + \sqrt{1 - \lambda} Y)) \leq $$
\begin{eqnarray*}
&& \leq \, \int \left( \int f^2(\sqrt \lambda x + \sqrt{1 - \lambda} y) \, d\P_X(x)\right) \, \log \left( \int f^2(\sqrt \lambda x + \sqrt{1 - \lambda} y) \, d\P_X(x)\right) \, dP_Y(y) \\ && \quad +  \int \left(C_{LS}(X) \, \int \, \lambda \, |\nabla f|^2(\sqrt \lambda x + \sqrt{1 - \lambda} y) \, d\P_X(x)\right) \, dP_Y(y) \\ && \leq \, \left(\int f^2(\sqrt \lambda x + \sqrt{1 - \lambda} y) \, d\P_X(x) \, d\P_Y(y)\right) \, \log \left( \int f^2(\sqrt \lambda x + \sqrt{1 - \lambda} y) \, d\P_X(x) \, d\P_Y(y) \right) \, \\ && \quad + \, C_{LS}(Y) \, \int \left|\nabla \sqrt{\int f^2(\sqrt \lambda x + \sqrt{1 - \lambda} y) \, d\P_X(x)} \right|^2 \, d\P_Y(y) \\ && \quad +  \lambda \, C_{LS}(X) \, \int \, |\nabla f|^2(\sqrt \lambda x + \sqrt{1 - \lambda} y) \, d\P_X(x) \, dP_Y(y)\\ && \leq  \, \E(f^2(\sqrt \lambda X + \sqrt{1 - \lambda} Y)) \, \log \left(\E(f^2(\sqrt \lambda X + \sqrt{1 - \lambda} Y))\right) \\ && \quad +  \lambda \, C_{LS}(X) \, \int \, |\nabla f|^2(\sqrt \lambda x + \sqrt{1 - \lambda} y) \, d\P_X(x) \, dP_Y(y) \\ && \quad + \, (1-\lambda) \, C_{LS}(Y) \, \int \left(\frac{\int |\nabla f^2|(\sqrt \lambda x + \sqrt{1 - \lambda} y) \, d\P_X(x)}{2 \sqrt{\int f^2(\sqrt \lambda x + \sqrt{1 - \lambda} y) \, d\P_X(x)}} \right)^2 \, d\P_Y(y) \, .
\end{eqnarray*}
Since $\nabla f^2=2 f \, \nabla f$, we may use Cauchy-Schwartz inequality in order to bound the last term in the latter sum. This yields exactly the desired result.

For the second statement, it is enough to use a change of variable, as we did in the gaussian case and the symmetry of $Y$. 
\end{proof}
Hence we get a general statement: \emph{the distribution of a random variable $X$ satisfies a Poincar\'e or a log-Sobolev inequality if and only if, for all or for one, symmetric random variable $Y$ independent of $X$ whose distribution satisfies a Poincar\'e or a log-Sobolev inequality, the distribution of $X+Y$ satisfies a Poincar\'e or a log-Sobolev inequality}. \hfill  $\diamondsuit$
\medskip

It is known that any log-concave probability measure satisfies some Poincar\'e inequality. The result is due to Bobkov \cite{bob99} (a short proof is contained in \cite{BBCG}). But if $Z$ is a log-concave random variable that do not satisfy a log-Sobolev inequality, $Z+\alpha G$ (which is still log-concave according to the Prekopa-Leindler theorem) does not satisfy a log-Sobolev inequality, in particular is not \emph{uniformly} log-concave.  
\medskip

\begin{remark}
Here is an amusing proof of the consequence of Prekopa's result when one variable is gaussian. 

Let $X$ (resp. $Y$) be a random variable with law $e^{-V(x)} dx$ (resp. a standard gaussian variable). We assume that $X$ and $Y$ are independent. The density of $X+\sqrt \lambda \, Y$ is thus given by $$q(x)= (2\pi \lambda)^{-n/2} \, \int e^{-V(u)} \, e^{-\frac{|x-u|^2}{2 \lambda}} \, du \, = \, (2\pi \lambda)^{-n/2} \, p(x) \, .$$ Let $H(x)$ be the hessian matrix of $\log p$. Then
\begin{eqnarray*}
 p^2(x) \, \langle \xi, H(x) \, \xi\rangle &=&   p(x) \, \frac{1}{\lambda^2} \, \left(\int \langle \xi,(x-u)\rangle^2 \, e^{-V(u)} \, e^{-\frac{|x-u|^2}{2 \lambda}} \, du\right) - \, \frac 1\lambda \, p^2(x) |\xi|^2\\ & & \, - \, \frac {1}{\lambda^2} \, \left( \int \langle\xi,(x-u)\rangle \,  \, e^{-V(u)} \, e^{-\frac{|x-u|^2}{2 \lambda}} \, du\right)^2 \, .
\end{eqnarray*} 
Hence
\begin{eqnarray*}
\langle \xi, H(x) \, \xi\rangle &=&  - \frac 1\lambda \, |\xi|^2  + \, \frac{1}{\lambda^2} \, \left(\int \langle \xi,(x-u)\rangle^2 \, e^{-V(u)} \, e^{-\frac{|x-u|^2}{2 \lambda}} \, \frac{du}{p(x)}\right) \\ & & \, - \, \frac{1}{\lambda^2} \, \left( \int \langle\xi,(x-u)\rangle \,  \, e^{-V(u)} \, e^{-\frac{|x-u|^2}{2 \lambda}} \, \frac{du}{p(x)}\right)^2 \, .
\end{eqnarray*} 
Now assume that $V$ satisfies (H.C.K) for some $K \in \R$. The probability measure $$e^{-V(u)} \, e^{-\frac{|x-u|^2}{2 \lambda}} \, \frac{du}{p(x)}$$(or if one prefers its potential) satisfies (H.C.$K+(1/\lambda)$). If $K+(1/\lambda)>0$, it thus satisfies a Poincar\'e inequality with constant $\lambda/(1+K\lambda)$. Applying this Poincar\'e inequality to the function $u \mapsto \langle \xi,(x-u)\rangle$, we obtain $$\langle \xi, H(x) \, \xi\rangle \leq - \, \frac{K}{1 + K \lambda} \, |\xi|^2 \, .$$ Thanks to simple scales we may thus state
\begin{proposition}\label{propprekopa}
Let $X$ be a random variable with law $e^{-V(x)} dx$ and $Y$ a standard gaussian variable independent of $X$. If $V$ satisfies (H.C.K) for $K \in \R$, then for $0\leq \lambda \leq 1$, the distribution of $\sqrt \lambda \, X + \sqrt{1-\lambda} \, Y$ satisfies (H.C.$\frac{K}{\lambda + K(1-\lambda)}$) as soon as  $\lambda + K(1-\lambda) > 0$. \\ In particular if $X$ is log-concave ($K=0$) so is $\sqrt \lambda \, X + \sqrt{1-\lambda} \, Y$.
\end{proposition}
\end{remark}
\bigskip

\appendix
\section{ Some general remarks.\\}\label{appendix}

We give in this section some general facts which we used from place to place.

First, if the flow $(\mu_t)_{t>0}$ of probability measures satisfies some uniform moment condition (i.e. $\sup_{t>0} \, \int |x|^p \, \mu_t(dx) < +\infty$ for some $p>1$) it is well known (Prohorov theorem) that it is weakly relatively compact. If they are absolutely continuous w.r.t. Lebesgue measure, Dunford-Pettis theorem gives us the existence of $\L^1$ weak limits for their densities. As a consequence if $W_p(\mu_t^x,\mu_t^y) \to 0$ as $t$ goes to infinity for all pair $(x,y)$ (or more generally for all pair of initial measures admitting a moment of order $p$), it is not difficult to see that there exists an unique limiting distribution for all initial distribution with a moment of order $p$, and, for linear diffusions that this distribution is an invariant probability measure. Using truncature one can deduce the uniqueness of an invariant measure. To obtain functional inequalities for the invariant probability measure $d\mu=e^{-U} dx$ (supposed to exist) when the semi-group $P_t$ is symmetric i.e. when $L = \Delta - \, \frac 12 \, \nabla U \, \nabla$, one can use convergence in Wasserstein distance.

First we recall two facts one can find for instance in \cite{cat4} remark 2.11 and in \cite{CGZ} lemma 2.12:
\begin{lemma}\label{lemsymgeneral}
If $P_t$ is $\mu$-symmetric  
\begin{enumerate}
\item \quad $t \mapsto \int \, |\nabla P_tf|^2 \, d\mu$ is non increasing.  Hence $\parallel \nabla P_tf\parallel_{\L^2(\mu)} \leq \frac{\sqrt 2}{\sqrt t} \, \parallel  f\parallel_{\L^2(\mu)} \, .$
\item \quad if there exists $\beta>0$ such that for all $f$ in a dense subset of $\L^2(\mu)$ there exists $c_f$ with $\Var_\mu(P_tf) \leq c_f \, e^{-\beta t}$ then $\Var_\mu(P_tf) \leq \, e^{-\beta t} \, \Var_\mu(f)$ for all $f \in \L^2(\mu)$.\\ Hence $\mu$ satisfies a Poincar\'e inequality with constant $C_P \leq 1/\beta$.
\end{enumerate}
\end{lemma}
Accordingly, since $\Var_\mu(P_tf) = \frac 12 \, \int_t^{+\infty} \, |\nabla P_sf|^2 \, ds$, a control  $|\nabla P_sf| \leq c_f \, e^{-\beta s}$ will furnish a Poincar\'e inequality. Notice that if $\parallel \nabla P_af\parallel_\infty \leq \rho \, \parallel \nabla f\parallel_\infty$ for some $a>0$, $\rho<1$ and all Lipschitz functions $f$, then using the semi-group property we get $\parallel \nabla P_tf\parallel_\infty \leq C \, \rho^t \, \parallel \nabla f\parallel_\infty$ for some $C>0$ and all $t$, hence a Poincar\'e inequality. The latter is a weaker form of the commutation of the gradient and the semi-group up to an exponential rate.

One can use the previous remarks to show that an exponential decay of Wasserstein distances furnishes some Poincar\'e inequality for $\mu$. In what follows $W_0$ denotes the total variation distance and $W_1$ is the usual $1$-Wasserstein distance. 
\begin{proposition}\label{propdecwas}
Assume that for all bounded (resp. Lipschitz) density of probability $h$ we have $W_0(P_th \mu,\mu) \leq c_h(t)$ (reps. $W_1$). Then for all bounded (resp. Lipschitz and bounded) $f$, there exist $c_f$ and $h$ such that $\Var_\mu(P_tf) \leq c_f \, c_h(2t)$. In particular if $c_h(t) = c_h \, e^{-\beta t}$, $\mu$ satisfies a Poincar\'e inequality.
\end{proposition}
\begin{proof}
Let $f$ be bounded and centered, and $$h=(f+\parallel f\parallel_\infty)/\int (f+\parallel f\parallel_\infty) \, d\mu=1 +(f/\parallel f\parallel_\infty) \, .$$ $h$ is thus a density of probability with $\parallel h\parallel_\infty \leq 2$. We have
\begin{eqnarray*}
\Var_\mu(P_tf) & = & \parallel f\parallel_\infty^2 \, \Var_\mu(P_th) \\ & \leq &   \parallel f\parallel_\infty^2 \, \int P_th (P_th-1) \, d\mu  =  \parallel f\parallel_\infty^2 \,  \int h (P_{2t}h-1) \, d\mu \\  & \leq & \parallel f\parallel_\infty^2 \, \parallel h\parallel_\infty \, W_0(P_{2t}h \mu,\mu) \, \leq \, 2 \, \parallel f\parallel_\infty^2 \, c_h(2t) \, .
\end{eqnarray*}
One can replace $W_0$ by $W_1$, just replacing $\parallel h\parallel_\infty$ by $\parallel \nabla h\parallel_\infty$ in which case $$\Var_\mu(P_tf) \leq  \, \parallel f\parallel_\infty \, \parallel \nabla f\parallel_\infty \, c_h(2t) \, .$$ 
\end{proof}

\begin{remark}\label{remnonsymW}
The previous result partly extends to the non symmetric situation of section \ref{secdiff}. Indeed if we do not use the symmetry of $P_t$, but only the fact that $\parallel P_t\parallel \leq 1$ in $\L^\infty$ we obtain that provided $W_0(P_th \mu,\mu) \leq c_h(t)$, $$\Var_\mu(P_tf) \leq 2 \, \parallel f\parallel_\infty^2 \,  c_h(t) \, . \quad \quad \quad \diamondsuit $$
\end{remark}
Even when the decay is not exponential, one gets a weak form of the Poincar\'e inequality (called a weak Poincar\'e inequality).
\begin{corollary}\label{corwpw}
In the situation of proposition \ref{propdecwas}, assume that $c_h(t)=c_h \, c(t)$ with $c(t) \to 0$ as $t\to +\infty$.  Then $$\Var_\mu(f) \, \leq \, \alpha(s) \, \int |\nabla f|^2 \, d\mu + \, s \, \Psi(f) \, ,$$ for all $s>0$ where $\Psi(f)=c_h \, \parallel f -\int f d\mu\parallel_\infty^2$ for $W_0$ and $\Psi(f)=c_h \, \parallel f-\int f d\mu\parallel_\infty \, \parallel \nabla f\parallel_\infty$ for $W_1$ and $\alpha(s) = s \, \inf_{u>0} \, \frac 1u \, c^{-1}(u \exp(1-(u/s)))$ .
\end{corollary}
\begin{proof}
Once we notice that the transformation $f \mapsto \lambda f$ does not change $h$ the result follows from \cite{rw} Theorem 2.3.
\end{proof}
\medskip

{\bf Acknowledgements}:  This research was supported by the French ANR project STAB.

\bibliographystyle{plain}
\bibliography{difflogconcaveKLS}

\end{document}